%% file: main.tex
\documentclass[12pt]{article} %***
\usepackage[sectionbib]{natbib}
\usepackage{array,epsfig,fancyheadings,rotating}
\usepackage{algorithm} 
\usepackage{algorithmicx,setspace}
\usepackage{algpseudocode} 
\usepackage[]{hyperref} %<----modified by Ivan
\usepackage{sectsty, secdot}
\sectionfont{\fontsize{12}{14pt plus.8pt minus .6pt}\selectfont}
\renewcommand{\theequation}{\thesection\arabic{equation}}
\subsectionfont{\fontsize{12}{14pt plus.8pt minus .6pt}\selectfont}
\usepackage{ulem}
\def\P{\mathbb P}

\def\R{\mathbb R}

\def\E{\mathbb E}
\def\vX{\boldsymbol{X}}
\def\mnull{\mathrm{null}}
\def\X{\boldsymbol X}

\def\l{\left(}
\def\r{\right)}
\def\lmi{\left[}
\def\rmi{\right]}

\def\lno{\left\|}
\def\rno{\right\|}
\def\bbeta{\boldsymbol{\beta}}
\def\lb{\left\{}
\def\rb{\right\}}
\def\var{\mathrm{var}}
\def\M{\boldsymbol M}
\def\x{\boldsymbol x}
\def\Z{\mathbb{Z}}
\def\m{\boldsymbol m}
\def\H{\mathcal H}
\def\S{\mathcal{S}_{Y|\bs{X}}}

\def\rank{\mathrm{rank}}

\def\mi{\mathrm{i}}
\newcommand{\mb}{\mathbb}
\newcommand{\mc}{\mathcal}
\newcommand{\bs}{\boldsymbol}
\newcommand{\wh}{\widehat}
\newcommand{\mr}{\mathrm}

\newcommand{\wt}{\widetilde}
\newcommand{\ve}{\varepsilon}

\newcommand{\ttE}{\mathtt{E}}
\newcommand{\Y}{\bs{Y}}
\usepackage{graphicx}
\usepackage{subfigure}
\usepackage{amsmath,color}
\usepackage{algorithmicx,algorithm}
\usepackage{amssymb}
\usepackage{amsfonts}
\usepackage{multirow}
\usepackage{amsthm}
\numberwithin{equation}{section}
\newtheorem{theorem}{Theorem}
\newtheorem{lemma}{Lemma}
\newtheorem{corollary}{Corollary}
\newtheorem{proposition}{Proposition}
\theoremstyle{definition}
\newtheorem{definition}{Definition}
\newtheorem{example}{Example}
\newtheorem{assumption}{Assumption}
\newtheorem{remark}{Remark}
\usepackage{hyperref}
\usepackage{cleveref}

\usepackage{mleftright}

\newcommand{\pt}[1]{\left(#1\right)}

\newcommand{\cl}[1]{\left\{#1\right\}}
\newcommand{\ang}[1]{\left\langle#1\right\rangle}

\newcommand{\norm}[1]{\left\lVert#1\right\rVert}

\newcommand{\mpt}[1]{\mleft(#1\mright)}

\usepackage{geometry}
\usepackage{authblk}
\title{Functional Slicing-free Inverse Regression via Martingale Difference Divergence Operator}
\author[a]{Songtao Tian}
\author[a]{Zixiong Yu\thanks{Co-first author.}}
\author[b]{Rui Chen\thanks{Corresponding author.}}
\affil[a]{Department of Mathematical Sciences, Tsinghua University}
\affil[b]{Center for Statistical Science, Department of Industrial Engineering, Tsinghua University}

\date{} % 去掉日期
\begin{document}
\maketitle
\renewcommand{\baselinestretch}{2}
\begin{quotation}
\begin{spacing}{2}
\noindent
\small Abstract:~
Functional sliced inverse regression (FSIR) is one of the most popular algorithms for functional sufficient dimension reduction (FSDR). However, the choice of slice scheme in FSIR is critical but challenging.
In this paper, we propose a new method called functional slicing-free inverse regression (FSFIR) to estimate the central subspace in FSDR. FSFIR is based on the martingale difference divergence operator, which is a novel metric introduced to characterize the conditional mean independence of a functional predictor on a multivariate response. We also provide a specific convergence rate for the FSFIR estimator.
Compared with existing functional sliced inverse regression methods, FSFIR does not require the selection of a slice number. Simulations demonstrate the efficiency and convenience of FSFIR.
\end{spacing}
\vspace{9pt}

\par
\end{quotation}\par

	\def\thefigure{\arabic{figure}}
	\def\thetable{\arabic{table}}
	
	\renewcommand{\theequation}{\thesection.\arabic{equation}}

	\fontsize{12}{14pt plus.8pt minus .6pt}\selectfont
	
% \tableofcontents
\newpage
\section{Introduction}

Classical statistical methods often failed in the high-dimensional data where the number of features $p$ is comparable to or even larger than the number of observed samples $n$. Sufficient dimension reduction (SDR)
is often the first step in dealing with high-dimensional problems. SDR aims to finding the minimal low-rank projection, 
of a predictor $\X\in\R^p$, which contains all the information of a response $Y\in\R$ without estimating the unknown link function. 
Indeed, SDR gives the intersection of all spaces $\mc S\subset \R^p$ satisfying $Y\perp \!\!\! \perp \bs X|P_{\mathcal S} \bs X$ where $\perp \!\!\! \perp$ denotes independence and
$P_{\mc S}$ denotes the projection onto $\mc S$. The said intersection is known as the \textit{central subspace} and denoted by $\mc S_{Y|\X}$. 

There are various well-known SDR methods: \textit{sliced inverse regression} (SIR, \citealt{li1991sliced}), 
 \textit{sliced average variance estimation} (SAVE, \citealt{cook1991sliced}), \textit{principal hessian directions} (PHD, \citealt{li1992principal}), \textit{directional regression} (DR, \cite{li2007directional}), \textit{minimum average variance estimation} (MAVE, \cite{xia2009adaptive}) and many others.
Among these methods that deal with the SDR problems,
SIR is particularly popular  due to its simplicity and efficiency. 
Under linearity and coverage conditions (See Assumption \ref{as:Linearity condition and Coverage condition}), SIR relates the central subspace $\S$ with the eigen-space of the covariance of conditional mean, i.e.,   $\var\mpt{\E[\X|Y]}$. Then SIR estimates $\var\mpt{\E[\X|Y]}$ by dividing the samples into several equally-sized slices according to the order statistics of response and averaging the sample covariance within each slice. 

However, the consistency and convergence rate of SIR, together with other slice-based SDR methods, involve the choice of a suitable slice number (denoted by $H$). 
\cite{hsing1992asymptotic}
 proved that SIR gives a root $n$ consistent estimation when each slice contains two observations (i.e.,  $H=n/2$), whereas \cite{zhu1995asymptotics}
argued that a smaller number of observations in each slice could yield a bigger covariance matrix of asymptotically multinormal distribution of estimators. In addition, by considering the case where the number of samples in each slice, denoted by $c$, goes to infinity with increasing sample size $n$,
 \cite{zhu1995asymptotics} showed that the optimal $c$ satisfies $c=c_0n^\alpha$ where $c_0>0$ is a constant and $\alpha\in\pt{0,1/2}$ depends on the distribution of the predictor $\X$ and the central curve $m(y):=\E[\bs X|Y=y]$.  Nevertheless, 
determining the constants $c_0$ and $\alpha$ in practice is usually challenging.

Notably, in the multivariate-valued predictor case, several approaches have been proposed to address the issue of selecting a suitable slice number $H$. For instance, 
\cite{zhu2010dimension} suggested
 \textit{cumulative slicing estimation} by considering all estimations with two slices for the central subspace $\S$. They proposed \textit{cumulative mean estimation}, \textit{cumulative variance estimation} and \textit{cumulative directional regression } parallel to SIR, SAVE and DR respectively.
\cite{cook2014fused} also proposed a \textit{fusing} method to relieve the burden of choosing a suitable $H$.
However, this method is not entirely slicing-free because it still requires a predefined collection of quantile slicing schemes. In order to provide an adaptive slicing scheme, \cite{wang2019dimension} implemented a regularization criteria by transforming the eigen-decomposition problem into a trace-optimization problem.
\cite{mai2021slicing} proposed a \textit{slicing-free inverse regression} method
 with the help of the \textit{martingale difference divergence matrix} (MDDM) which measures the conditional mean independence of a high-dimensional predictor on a multivariate response. 

It is an important trend to study  functional data analysis in SDR and many significant achievements have been made in this area. For example, 
\cite{ferre2003functional} first applied SIR to functional-valued data where $\bs X$ is in $\mathcal L_2([0,1])$ (the separable Hilbert space of square-integrable curves on $[0,1]$) and the response $Y$ is in $\mb R$. Based on this work, various developments in FSIR have been made (e.g., \cite{Forzanicook2007note,lian2014Sefsdr,lian2015functional,wang2020functional,chen2023optimality}).

 Again, the choice of the slicing number $H$  in FSIR remains an issue. In addition, although the central space can be defined for both univariate and multivariate responses, most existing SDR methods, such as FSIR , primarily focus on the case of univariate response. Extending these methods to handle multivariate response is a non-trivial task, and this limitation is evident in FSIR.

Therefore, it is natural to seek for a \textit{functional slicing-free} method for  multivariate response to avoid the difficulty in choosing a suitable $H$.
 In this paper, we propose a new method we call functional slicing-free inverse regression (FSFIR) to estimate the central subspace $\mc S_{\Y|\bs X}$ for $(\vX,\Y)\in\mathcal L_2([0,1])\times\R^q$ without specifying any $H$.  
 
\subsection{Major contributions}
The FSFIR method is our solution to the aforementioned goals. Around this core innovation, the main clues and results of our article are explained as follows.

 First, we introduce a new metric: martingale difference divergence operator
 (MDDO), which generalizes MDDM in \cite{lee2018martingale}. It turns out that MDDO enjoys a lot of properties similar to MDDM: 
 \begin{itemize}
 \item[$\mathrm{(i)}$]$\mathrm{MDDO}(\boldsymbol{X}|\Y)=0\Longleftrightarrow \mathbb E[\boldsymbol{X}|\Y]=0$\quad a.s.;
\item[$\mathrm{(ii)}$] $\mathrm{MDDO}(T^*{\boldsymbol{X}}|\Y)=T^*\mathrm{MDDO}({\boldsymbol{X}}|\Y)T,\quad\forall T:\mc H\to\mc H $.
 \end{itemize}
 From (i) we see that MDDO quantifies the conditional mean independence of a functional predictor $\bs X$ on a multivariate response $\Y$.
 It further implies that under the linearity and coverage conditions, the central subspace $\mc S_{\Y|\vX}$ and the image of MDDO are closely related:
 \begin{equation}\label{Eqt:GammaSMDDO}
 \Gamma \mc S_{\Y|\vX}=\mathrm{Im}\{\mathrm{MDDO}(\boldsymbol{X}|\Y)\}.
 \end{equation}
 
 Second, we propose the FSFIR by estimating $\mathrm{MDDO}(\boldsymbol{X}|\Y)$ without specifying any $H$.  Our method is based on the above \eqref{Eqt:GammaSMDDO} and inspired by  \cite{mai2021slicing}.
 To tackle the issue that $\Gamma^{-1}$ may be unbounded, we adopt a truncation on the predictor $\X$.

 Finally, we derive a specific convergence rate of FSFIR for estimating the central subspace $\mc S_{\Y|\vX}$. To compare FSFIR with classical FSIR methods including truncated FSIR \citep{ferre2003functional,chen2023optimality} and regularized FSIR \citep{lian2015functional}, we conduct simulations contains the 
subspace estimation error performance of FSFIR on both synthetic data  and real data. The results demonstrate the good performance and convenience of our FSFIR compared with   FSIR methods.

The rest of this paper is organized as follows. In Section $\ref{subsection, review}$, we review MDDM briefly. Detailed definition and properties of MDDO are in Section $\ref{subsection, MDDO}$. 
 Section $\ref{section, Slicing-free IR via MDDO}$ establishes 
the medium to estimate the central subspace $\mc S_{\Y|\vX}$ in terms of MDDO, i.e.,  Equation \eqref{Eqt:GammaSMDDO}. In Section $\ref{subsection, estimation method}$, we propose the FSFIR for estimating the central subspace $\mc S_{\Y|\vX}$ and then design a detailed algorithm for FSFIR. The specific convergence rate of FSFIR is given in Section $\ref{subsection, convergence rate}$. 
Section $\ref{section, experiments}$
contains our experiments. We make some concluding remarks in  Section $\ref{section, Discussion}$.

\subsection*{Notations}
Let $\mathcal H:=\mathcal L_2([0,1])$ be the separable Hilbert space of square-integrable curves on $[0,1]$ with the inner product $\ang{f,g}=\displaystyle{\int_{0}^1}f(u)g(u)\,\mathrm{d}u$ and norm $\norm{f}:=\sqrt{\ang{f,f}}$ for $f,g\in\mathcal H$.

Given any
operator $T$ on $\mathcal H$, we use $\mathrm{Im}(T)$ and $\mathrm{null}(T)$
to denote  the closure of  image of $T$, and the null space of $T$ respectively. Besides,  we use $P_T$ to denote the projection operator from $\mc H$ to ${\mathrm{Im}}(T)$,  $T^*$ the adjoint operator of $T$ (a bounded linear operator).
We use $\norm{T}$ to denote the operator norm  with respect to $\ang{\cdot,\cdot}$ of $T$:
\begin{align*}
\norm{T} :={\sup_{\bs{\beta}\in\mathbb{S}_\mathcal{H}}}\|T(\bs{\beta})\|
\end{align*}
where $\mathbb{S}_{\mathcal{H}}=\cl{\bs{\beta}\in\mathcal{H}:\norm{\bs{\beta}}=1}$.
If $T$ is further compact, then we use $\sigma_j(T)$ to denote the $j$-th singular value of $T$. 
 When $T$ is  positive semi-definite and compact, 
  $T^\dagger$  denotes the  Moore-Penrose pseudo-inverse of $T$ and 
 $\lambda_j(T)$ denotes the $j$-th eigenvalue of $T$.
Abusing notations, we also denote by $P_S$ the projection operator onto a closed space $S\subseteq \mathcal{H}$.
% We use $\mathrm{spec}(T)$ to denote the set of all eigenvalues of the operator $T$. 
 For any $x,y\in\H$, their tensor product $x\otimes y:\H\to\H$ is defined to be the linear operator: $(x\otimes y)(z)=\ang{x,z}y$ for all $z\in\H$.
 For any random element $\boldsymbol{X}=\bs X_t\in \H$, its mean function is defined as $(\mb E\bs X)_t=\mb E[\bs X_t]$.  
%  the mean $\E\boldsymbol{X}$ is defined as the unique element in $\H$ such that for all $z\in \H$, 
% $\langle \E\boldsymbol{X}, z \rangle=\E\langle \boldsymbol{X}, z \rangle$. 
  For any random operator $T$ on $\H$, the mean $\E[T]$ is defined as the unique operator on $\H$ such that for all $z\in \H$, 
$ (\E[T])(z)=\E[T(z)]$. Specifically,
% The covariance operator of $\boldsymbol{X}$, $\var(\boldsymbol{X})$, is defined as $\var(\boldsymbol{X})(z)=\E\left( \langle \boldsymbol{X}, z \rangle \boldsymbol{X} \right) - \langle\E\boldsymbol{X}, z \rangle \E\boldsymbol{X}$. 
we denote by $\Gamma$ the covariance operator of $\X$, i.e.,  
\begin{align*}
\Gamma:=\var(\bs X)=\E[(\X-\E\X)\otimes (\X-\E\X)] 
\end{align*}
satisfying  $\var(\boldsymbol{X})(z)=\E\mpt{ \ang{\boldsymbol{X}, z} \boldsymbol{X}}- \ang{\E\boldsymbol{X}, z}\E\boldsymbol{X}$.

Throughout the paper, $C_i$ stands for a generic 
constant, $i\geqslant 0$ being an integer. Note that $C_i$ depends on the context. 
For a random sequence $X_n$, we denote by $X_n=O_{\mathbb{P}}(a_n)$ that $\forall\varepsilon>0$, there exists a constant $C_\varepsilon>0$, such that
$\sup_n\mathbb{P}(|X_n|\geqslant C_\varepsilon a_n)\leqslant\varepsilon$.  
For two sequences $a_n$ and $b_n$, we denote $a_n\lesssim  b_n$  if there
exists  a  positive constant $C$ such that $a_n\leqslant Cb_n$. Let $[k]$ denote $\{1,2,\dots,k\}$ for some positive integer $k\geqslant1$.
% We hide the log factors of $n$ in $\widetilde O_{\mathbb{P}}$.
% We denoted by $X_n=o_p(1)$ that $\{X_n\}$ converges in probability to $0$.

\section{MDDO for Functional Data}\label{section, MDDO for functional data}

In this section, we introduce a new metric which we call  martingale difference divergence operator (MDDO) to measure the conditional mean independence
of a functional-valued predictor on  a multivariate response. We are mainly motivated by the work \cite{lee2018martingale} that introduced the notion of  martingale difference divergence matrix (MDDM).
\subsection{Review of the MDDM}\label{subsection, review}
To characterize the conditional mean independence of $\bs V:=(V_1,...,V_p)^\top\in\mb R^p$ on $\bs U=(U_1,...,U_q)^\top\in\mb R^q$, \cite{lee2018martingale}
define the MDDM, which we now recall.
\begin{definition} [\citealt{lee2018martingale}]
\label{def: MDDM}
For $\bs V\in\mb R^p$ and $\bs U\in\mb R^q$, let $$H(\bs s):=(H_1(\bs s),...,H_p(\bs s))^\top\in \mb R^p\qquad~\text{for}~\bs s\in\mb R^q$$ where $H_j(\bs s)$ is defined by $H_j(\bs s)=\mathrm{cov}(V_j,e^{\mi\langle \bs s,\bs U\rangle})$, $\mi=\sqrt{-1}$. Let $H^*(\bs s)$ be the conjugate-transpose of $H(\bs s)$. Then the following matrix is called the Martingale Difference Divergence Matrix (MDDM):
\begin{equation*}
\mathrm{MDDM}(\bs V|\bs U):=\frac{1}{c_q}\int_{\mb R^q}\frac{H(\bs s)H^*(\bs s)}{\|\bs s\|^{1+q}}~\mathrm{d}\bs s,
\end{equation*}
where $c_q$ stands for the constant $\frac{\pi^{(q+1)/2}}{\varGamma\l (q+1)/2\r}$ with $\varGamma(\cdot)$ being the Gamma function. 
\end{definition}
\begin{remark}
The integration on $\mb{R}^q$ is taken in the sense of principal value, i.e.,  $\displaystyle{\int_{\mb{R}^q}}=\displaystyle{\lim\limits_{\varepsilon\to0^+}\int_{D_\varepsilon}}$, where $D_\varepsilon=\{\bs x\in\mb{R}^q:\varepsilon\leqslant\|\bs x\|\leqslant \varepsilon^{-1}\}$.
\end{remark}
Clearly, $\mathrm{MDDM}(\bs V|\bs U)$ is a positive semi-definite matrix. Suppose that $\mb E[\|\bs V\|^2+\|\bs U\|^2]<\infty$, then there is a simpler expression for MDDM:
\begin{equation*}
\mathrm{MDDM}(\bs V|\bs U)=- \mb E[(\bs V-\E\bs V)(\bs V'-\E\bs V')^\top\|\bs U-\bs U'\|],
\end{equation*}
where $(\bs V',\bs U')$ is another independent identical distributed (i.i.d.) copy of $(\bs V,\bs U)$. MDDM enjoys some properties:
\begin{proposition}[\citealt{lee2018martingale}]
\label{proposition, MDDM basic peoperties}~
\begin{itemize}
\item [$\mathrm{(i)}$]
For $\bs V\in\mb R^p$ and $\bs U\in\mb R^q$, $\bs V$ is conditional mean independent on $\bs U$ almost surely (a.s.) if and only if $\mathrm{MDDM}(\bs V|\bs U)$ vanishes, i.e., 
\[\mb E[\bs V|\bs U]=\mb E[\bs V]~\text{a.s.}\Longleftrightarrow \mathrm{MDDM}(\bs V|\bs U)=\bs{O}\]
where $\bs{O}$ stands for the zero matrix;
\item [$\mathrm{(ii)}$] For any $\bs A\in\mb R^{p\times d}$, we have $\mathrm{MDDM}(\bs A^\top\bs V|\bs U)=\bs A^\top\mathrm{MDDM}(\bs V|\bs U)\bs A$.
\end{itemize}
\end{proposition}
MDDM generalizes the \textit{Martingale Difference Divergence} (MDD) and the normalized MDD, namely \textit{Martingale Difference correlation} (MDC). For more about MDD and MDC, see \cite{shao2014martingale}. All these statistics can be used to measure the conditional mean independence of $\bs V$ on $\bs U$. Specifically, we have
\begin{align*}
\mb E[\bs V|\bs U]=\mb E[\bs V]~\text{a.s.}&\Longleftrightarrow\mathrm{MDD}(\bs V|\bs U)=\mathrm{MDC}(\bs V|\bs U)=0\\
&\Longleftrightarrow\mathrm{MDDM}(\bs V|\bs U)=\bs O.
\end{align*}
Meanwhile, we can estimate the MDDM in the following way: assume that we have observed $n$ i.i.d. samples $\{(\bs V_k,\bs U_k)\}_{k=1}^n$ from the same joint distribution as $(\bs V,\bs U)$. Then the finite sample counterpart of MDDM can be defined as 
\begin{equation*}
\mathrm{MDDM}_n(\bs V|\bs U):=- \frac{1}{n^2}\sum_{h,l=1}^n(\bs V_h-\overline{\bs V}_n)(\bs V_l-\overline {\bs V}_n)^\top\|\bs U_h-\bs U_l\|,
\end{equation*}
where $\overline {\bs V}_n:=n^{-1} \sum_{i=1}^n \bs V_i$ is the sample mean of $\{\bs V_i\}_{i=1}^n$.
\cite{mai2021slicing} related the eigenspace (i.e.,  the space spanned by eigenfunctions corresponding to the nonzero eigenvalues) of MDDM with the central subspace in SDR and then proposed a slicing-free inverse regression estimator. Furthermore, they proved a key large deviation inequality
between $\mathrm{MDDM}_n(\bs V|\bs U)$ and $\mathrm{MDDM}(\bs V|\bs U)$ and then established the consistency of this estimator.
\subsection{MDDO}\label{subsection, MDDO}
In this section, we will generalize the MDDM in Definition $\ref{def: MDDM}$ to the MDDO. Roughly speaking, we replace the vector-valued $(\bs V,\bs U)\in\R^{p}\times\R^q$ with a functional-valued $(\bs X,\Y)\in\mc H\times \R^{q}$. Accordingly, we need the tensor product to replace the matrix product. 

Without loss of generality, we assume that $\vX\in\mc H$ satisfies $\mathbb{E}[\boldsymbol{X}]=0$ throughout the paper. As is usually done in functional data analysis \citep{ferre2003functional,lian2014Sefsdr,lian2015functional,chen2023optimality}, we assume  that $\mathbb{E}[\|\boldsymbol{X}\|^4]<\infty$, 
which implies
that $\Gamma$ is a trace class \citep{hsing2015theoretical} and $\boldsymbol{X}$ possesses the following Karhunen--Lo\'{e}ve expansion: 
 there exists a pairwise uncorrelated random sequence $\{\omega_j\}_{j\in\Z_{\geqslant1}}$ with $\mb E[\omega_j]=0$ and $\mathrm{var}(\omega_j)=1$ for all $j$, such that 
\begin{equation}\label{eq:X expansion}
\bs X=\sum_{j=1}^\infty \sqrt{\lambda_j}\omega_j\phi_j,
\end{equation}
where $\{\lambda_j\}_{j\in\Z_{\geqslant1}}$ and $\{\phi_j\}_{j\in\Z_{\geqslant1}}$ are eigenvalues (with descent order) and associated  eigenfunctions of $\Gamma$.
In addition, we assume that $\Gamma$ is non-singular (i.e., $\lambda_i>0, \forall i$) as the literature on functional data analysis usually does. Since $\Gamma$ is compact ($\Gamma$ is a trace class), by spectral decomposition theorem of compact operators, we know that $\{\phi_{i}\}^{\infty}_{i=1}$ forms a complete basis of $\mc H$.

% \begin{equation}\label{eq:X expansion}
% \bs X=\sum_{j=1}^\infty \sqrt{\lambda_j}\omega_j\phi_j,
% \end{equation}
% where $\xi_{i}$'s  are random variables satisfying   $\E[\xi^{2}_{i}]=\lambda_{i}$  and $\E[\xi_{i}\xi_{j}]=0$ for $i\neq j$ and $\{\phi_{i}\}^{\infty}_{i=1}$ are the eigenfunctions of  $\Gamma$ associated with  the decreasing eigenvalues sequence  $\{\lambda_{i}\}^{\infty}_{i=1}$. 
% In addition, we assume that $\Gamma$ is non-singular (i.e., $\lambda_i>0, \forall i$) as the literature on functional data analysis usually does. Since $\Gamma$ is compact ($\Gamma$ is a trace class), by spectral decomposition theorem of compact operators, we know that $\{\phi_{i}\}^{\infty}_{i=1}$ forms a complete basis of $\mc H$. 
% Without loss of generality, we assume that $\mathbb{E}[\boldsymbol{X}]=0$ throughout the paper. As is usually done in functional data analysis \citep{ferre2003functional,lian2014Sefsdr,lian2015functional,chen2023optimality}, we  assume that $\mathbb{E}[\|\boldsymbol{X}\|^4]<\infty$, which implies that $\Gamma$ is a trace class and $\boldsymbol{X}$ possesses the following Karhunen–Lo\'{e}ve expansion:  there exists a pairwise uncorrelated random sequence $\{\omega_j\}_{j\in\Z_{\geqslant1}}$ with $\mb E[\omega_j]=0$ and $\mathrm{var}(\omega_j)=1$ for all $j$, such that 
% \begin{equation}\label{eq:X expansion}
% \bs X=\sum_{j=1}^\infty \sqrt{\lambda_j}\omega_j\phi_j,
% \end{equation}
% where $\{\lambda_j\}_{j\in\Z_{\geqslant1}}$ and $\{\phi_j\}_{j\in\Z_{\geqslant1}}$ are eigenvalues (with descent order) and associated  eigenfunctions of $\Gamma$.

Now, we can state our definition of MDDO. 
% \begin{definition}\label{def: MDDO}
% For $(\bs X,Y)\in\mc H\times \mb R$, we define the Martingale Difference Divergence Operator (MDDO) by 
% \begin{equation*}
% \mathrm{MDDO}(\bs X|Y):=\frac{1}{\pi}\int_{\mb R}\frac{G_{s}\otimes \overline G_{s}}{s^{2}}~\mathrm{d} s,
% \end{equation*}
% where $G_{s}\in\mc H$ is defined as $G_{s}(t)=\mathrm{cov}\hspace{-0.9mm}\left(\bs X(t),e^{\mi sY}\right)$ for any $t\in[0,1]$ and $\overline G_{s}$ is the complex
% conjugate of $G_{s}$.
% \end{definition}

\begin{definition}\label{def: MDDO}
For $\bs X\in\mc H$ and $\Y\in\R^q$, we define the Martingale Difference Divergence Operator (MDDO) by 
\begin{equation*}
\mathrm{MDDO}(\bs X|\Y):=\frac{1}{c_q}\int_{\mb R^q}\frac{G_{\bs s}\otimes \overline G_{\bs s}}{\|\bs s\|^{1+q}}~\mathrm{d}\bs s,
\end{equation*}
where $G_{\bs s}\in\mc H$ for $\bs s\in\R^q$ is defined as $G_{\bs s}(t)=\mathrm{cov}\hspace{-0.9mm}\left(\bs X(t),e^{\mi \langle\bs s,\Y\rangle}\right)$ for any $t\in[0,1]$ and $\overline G_{\bs s}$ is the complex
conjugate of $G_{\bs s}$.
\end{definition}

\begin{remark}
Again, the integration on $\mb{R}^q$ is taken in the sense of principal value.
\end{remark}

Clearly, $\mathrm{MDDO}(\bs X|\Y)$ is a positive semi-definite operator from $\mc H$ to itself. Next we characterize MDDO in terms of expectation, 
which is easier to compute.
% i.e.,  eigenvalues of $\mathrm{MDDO}(\bs X|Y)$ are all real and non-negative. 
In order to achieve this, we need a global
assumption:
\begin{assumption}\label{as:joint distribution assumption}~
 The joint distribution of $(\boldsymbol{X}, \Y)\in\mathcal H\times \mathbb{R}^q$ satisfies the second-order-moment condition, i.e.,  
$\mb E[\|\bs X\|^2+\|\Y\|^2]<\infty$.
\end{assumption}
\begin{lemma}\label{lemma, equivalence of two def of MDDO}Under Assumption $\ref{as:joint distribution assumption}$, we have,
\begin{equation*}
\mathrm{MDDO}(\bs X|\Y)=- \mb E[\bs X\otimes\bs X'\|\Y-\Y'\|],
\end{equation*}
where $(\bs X',\Y')$ is an i.i.d. copy of $(\bs X,\Y)$.
\end{lemma}
Similar  
simple expressions in terms of expectation have been established for MDD in \cite{shao2014martingale}, and MDDM in \cite{lee2018martingale}. 
Next, we show some properties of MDDO which will be used in Sections \ref{section, Slicing-free IR via MDDO} and \ref{section, error}.
\begin{theorem}\label{theorem, MDDO and conditional mean independence}
Under Assumption $\ref{as:joint distribution assumption}$, the following facts hold:
\begin{itemize}
\item[$\mathrm{(i)}$]$\mathrm{MDDO}(\boldsymbol{X}|\Y)=0\Longleftrightarrow \mathbb E[\boldsymbol{X}|\Y]=0$\quad a.s.;
\item[$\mathrm{(ii)}$] $\mathrm{MDDO}(T^*{\boldsymbol{X}}|\Y)=T^*\mathrm{MDDO}({\boldsymbol{X}}|\Y)T,\quad\forall T:\mc H\to\mc H $.
\end{itemize}
\end{theorem}
These properties generalize Proposition $\ref{proposition, MDDM basic peoperties}$.
Property (i) in Theorem \ref{theorem, MDDO and conditional mean independence} means MDDO can characterize the conditional mean independence of a functional-valued predictor on a multivariate response. Based on (i), we relate the central subspace $\mc S_{\Y|\vX}$ with the image of $\mathrm{MDDO}(\bs X|\Y)$, see the next Section $\ref{section, Slicing-free IR via MDDO}$.
\section{FSFIR via MDDO}\label{section, Slicing-free IR via MDDO}
\subsection{Review of  FSIR}\label{subsection, FSIR}

% {\color{blue}Most existing SIR
% literature  considers the following
%  multiple-index model with univariate response:}
% \begin{align}\label{eq:multiple index model}
% Y=f(\langle\bs{\beta}_1,\bs X\rangle,\dots,\langle\bs{\beta}_d,\bs X\rangle,\varepsilon), 
% \end{align}
% where $\bs X\in\mc H,\bs{\beta}_i\in\H~(i\in[d])$, 
% $f:\R^{d+1}\to \R$  is an unknown link function and $\varepsilon$ is a random noise independent of $\X$. Inspired by the advanced work of \cite{li1991sliced}, \cite{ferre2003functional} proposed the FSIR to estimate the central subspace
% \begin{align}\label{def: central subspace}\mc S_{Y|\bs{X}}=\mathrm{span}\{\bs{\beta}_1,...,\bs{\beta}_d\}.
% \end{align}
% Although $\bs{\beta}_i$'s are not individually identifiable  in view of the existence of unknown link function $f$, estimation of  $\{\bs{\beta}_1,...,\bs{\beta}_d\}$ is possible.

Functional SDR aims to  give the intersection of all spaces $\mc S\subset\mc H$ satisfying $\Y\perp \!\!\! \perp \bs X|P_{\mathcal S} \bs X$ for 
$(\bs X,\Y)\in\mc H\times \R^{q}$. The said intersection is known as the   \textit{functional central subspace} and denoted by $\mc S_{\Y|\X}$. To estimate $\mc S_{\Y|\X}$, people often need some mild conditions on $(\X,\Y)$ explained below. 

Assume that $\mc S_{\Y|\X}$ has a basis as follows:
\begin{align}\label{def: central subspace}\mc S_{\Y|\bs{X}}=\mathrm{span}\{\bs{\beta}_1,...,\bs{\beta}_d\}.
\end{align}
\begin{assumption}\label{as:Linearity condition and Coverage condition}
The joint distribution of $(\X,\Y)\in\mc H\times \R^{q}$ satisfies
\begin{itemize}
 \item[\textbf{i)}] \textit{Linearity condition:}
 The conditional expectation $\mathbb E\left[\langle \bs b,{\boldsymbol{X}}\rangle|\langle\bs{\beta}_1,{\boldsymbol{X}}\rangle,\dots,\langle\bs{\beta}_d,{\boldsymbol{X}}\rangle\right]$ is linear in $\langle \bs{\beta}_1,{\boldsymbol{X}}\rangle,...,\langle\bs{\beta}_d,{\boldsymbol{X}}\rangle$ for any $\bs b\in\mathcal H$.
 \item[\textbf{ii)}]\textit{Coverage condition:} $\mathrm{Rank}\l\mathrm{var}(\mb E[\bs X|\Y])\r=d$.
\end{itemize}
\end{assumption}
Both these conditions generalize multivariate ones in SDR literature \citep{li1991sliced,zhu2006sliced,lin2018consistency,lin2021optimality,mai2021slicing,huang2023sliced}. 

The next Lemma $\ref{as:Linearity condition and Coverage condition}$ is the basis of many functional SDR methods such as FSIR. 
\begin{lemma}[\citealt{ferre2003functional}]\label{lemma: SE=GammaS} 
Define
\begin{align}\label{equation, IRS}
\mathcal S_{\mathbb E(\boldsymbol{X}|\Y)}:=\mathrm{span}\{\mathbb E(\boldsymbol{X}|\Y=\bs y)\mid \bs y\in\mb R^q\}.
\end{align}
Then under Assumption $\ref{as:Linearity condition and Coverage condition}$, one has
\begin{equation*}
\Gamma \mc S_{\Y|\bs X}=\mc S_{\mathbb E(\boldsymbol{X}|\Y)}=\mathrm{Im}\{\mathrm{var}(\mb E(\bs X|\Y))\},
\end{equation*}
where $\Gamma \mc S_{\Y|\bs X}:=\mathrm{span}\{\Gamma\bs{\beta}_1,...,\Gamma\bs{\beta}_d\}.$
\end{lemma}
Based on this, one can estimate the central subspace $\mc S_{\Y|\vX}$ by estimating the eigenspace of $\Gamma^{-1} \mathrm{var}(\mb E[\bs X|\Y])$. 
While the central space is well-defined for both univariate and multivariate responses, most existing SDR methods, such as FSIR \citep{ferre2003functional}, mainly focus on the case of univariate response. 
Extending these methods to handle multivariate response is a challenging task.
The FSIR  procedures for estimating  $\mathrm{var}(\mb E[\bs X|Y])$ with univariate response can be briefly summarized as follows. Given $n$ i.i.d. samples $\{(\bs X_i,Y_{i})\}_{i\in[n]}$, 
% with the same distribution as $(\X,Y)$ from the multiple-index model \eqref{eq:multiple index model}
 FSIR divides the samples into $H$ equally-sized slices according to the order statistics $Y_{(i)}$ and estimates $\mathrm{var}[\mb E(\bs X|Y)]$ by
\begin{equation}
 \dfrac1H\sum\limits_{h=1}^H\overline{\bs {X}}_{h,\cdot}\otimes\overline{\bs {X}}_{h,\cdot},
\end{equation}
where $\overline{\bs {X}}_{h,\cdot}$ is the sample mean of the $h$-th slice. 

In general, the consistency and convergence rate
of SIR depend on the slice number $H$. Choosing a suitable slice number $H$ in SIR is a difficulty. 
 On one hand, too small $H$ yields too much samples in each slice. Then SIR can not fully characterize the dependence of the predictor on the response, which will cause a large bias. On the other hand, a small amount of samples in each slice yields a large variance of the estimation \citep{zhu1995asymptotics}.
To avoid this difficulty, we propose the method of FSFIR for multivariate response  in the following section (see also section $\ref{subsection, estimation method}$).
\subsection{Principle of FSFIR}\label{subsection, MDDO in SDR}
In this section, we lay the foundation of FSFIR.
The following conclusion about MDDO is needed.
\begin{theorem}\label{theorem, MDDO and IRS}
Under Assumptions $\ref{as:joint distribution assumption}$ and $\ref{as:Linearity condition and Coverage condition}$, we have
\[
{\mathcal{S}_{\mathbb E(\boldsymbol{X}|\Y)}}={\mathrm{Im}}\{\mathrm{MDDO}(\boldsymbol{X}|\Y)\}.
\]
\end{theorem}
This result generalizes \cite[Proposition 2]{mai2021slicing}.
Combining Theorem $\ref{theorem, MDDO and IRS}$ with Lemma $\ref{lemma: SE=GammaS}$, we can derive that:
\begin{corollary}
\label{corollary, MDDO and central subspace}
Under Assumptions $\ref{as:joint distribution assumption}$ and $\ref{as:Linearity condition and Coverage condition}$,
we have
\[\Gamma \mc S_{\Y|\vX}={\mathcal{S}_{\mathbb E(\boldsymbol{X}|\Y)}}=\mathrm{Im}\{\mathrm{MDDO}(\boldsymbol{X}|\Y)\}.\]
\end{corollary}
\section{FSFIR Algorithm and Corresponding Asymptotic Properties}\label{section, error}
Based on the above Corollary $\ref{corollary, MDDO and central subspace}$, we can estimate the central subspace $\mc S_{\Y|\vX}$ by estimating the eigen-space of $\Gamma^{- 1}\mathrm{MDDO}(\bs X|\Y)$ without specifying any slice scheme. So in this part, we develop the procedures of FSFIR in Section $\ref{subsection, estimation method}$ where a truncation scheme is adopted due to the unboundness of $\Gamma^{-1}$. 
Then in Section $\ref{subsection, convergence rate}$ we give the specific convergence rate of FSFIR for estimating $\mc S_{\Y|\vX}$. 
\subsection{FSFIR algorithm}\label{subsection, estimation method}
% In general, it is impossible to estimate $\Gamma^{-1}$ via $\wh\Gamma^{\dagger}$, the pseudo-inverse of the sample covariance operator $\widehat{\Gamma}:=\frac{1}{n}\sum^{n}_{i=1}\boldsymbol{X}_{i}\otimes \boldsymbol{X}_{i}$, because $\Gamma$ is compact ($\Gamma$ is a trace class) and $\Gamma^{-1}$ is unbounded. 
In general, it is not possible to estimate $\Gamma^{-1}$ directly using $\wh\Gamma^{\dagger}$, which is the pseudo-inverse of the sample covariance operator $\widehat{\Gamma}:=\frac{1}{n}\sum^{n}_{i=1}\boldsymbol{X}_{i}\otimes \boldsymbol{X}_{i}$. This is because $\Gamma$ is a compact operator (trace class) and $\Gamma^{-1}$ is unbounded.
% To tackle this issue, people adopted some measures 
To address this issue, various approaches have been proposed, such as truncation on the covariance operator \citep{ferre2003functional,chen2023optimality} and ridge-type regularization \citep{lian2015functional}. 
Our way to overcome this issue is to do truncation on the predictor (see also \cite{li2010deciding}).

% Before we elaborate on it, we need an assumption about decomposition of the random function $\bs X$:
% \begin{assumption}\label{assumption: decomposition} There exists a pairwise uncorrelated random sequence $\{\omega_j\}_{j\in\Z_{\geqslant1}}$ with $\mb E[\omega_j]=0$ and $\mathrm{var}(\omega_j)=1$ for all $j$, such that 
% \begin{equation}\label{eq:X expansion}
% \bs X=\sum_{j=1}^\infty \sqrt{\lambda_j}\omega_j\phi_j,
% \end{equation}
% where $\{\lambda_j\}_{j\in\Z_{\geqslant1}}$ is a descending sequence of real numbers and $\{\phi_j\}_{j\in\Z_{\geqslant1}}$ is an orthonormal basis of $\mc H$. 
% \end{assumption}
% {\color{red}
% \begin{remark}
% In fact, when $\bs X$ is a square integrable function with continuous covariance, Assumption $\ref{assumption: decomposition}$ coincides with the Karhunen–Lo\.{e}ve decomposition \citep{hsing2015theoretical}. 
% Furthermore,
% if $\boldsymbol{X}$ is a centered Gaussian process, Assumption $\ref{assumption: decomposition}$ holds with $\{\omega_j\}_{j\in\Z_{\geqslant1}}$ independent and standard normal random variables \citep{kac1947explicit,deheuvels2008karhunen}.
% \end{remark}}
Now we state our truncation scheme. For a smoothing parameter $m$ satisfying 
\begin{align}\label{eq: m n relationship}
m=n^{c_1},\quad c_1\in(0,1), 
\end{align}
we define the truncated predictor as follows:
\begin{align*}
\X^{(m)}:=\Pi_m \bs X=\sum\limits_{j=1}^m\sqrt{\lambda_j}\omega_j\phi_j
\end{align*}
where $\Pi_m:=\sum\limits_{i=1}^m\phi_{i}\otimes\phi_i$. 
Accordingly, the truncated covariance operator $\Gamma_m:=\mathrm{var}(\bs X^{(m)})$ and the pseudo-inverse of $\Gamma_m$ satisfy:
\begin{align*}
\Gamma_m&=\sum\limits_{i=1}^m\lambda_i\phi_i\otimes\phi_i\quad\text{and}\quad \Gamma^\dag_m=\sum\limits_{i=1}^m\lambda_i^{-1}\phi_i\otimes\phi_i
\end{align*}
respectively. 
% {\color{red}Clearly, $\|\Gamma-\Gamma_m\|$ tends to $0$ as $m\to\infty$. Thus we can approximate $\Gamma^{-1}$ by $\wh\Gamma^\dag_m$:
% \begin{align*}
% \widehat{\Gamma}_m:=\frac{1}{n}\sum_{i=1}^n \bs X_i^{(m)}\otimes \bs X_i^{(m)}
% \end{align*}
% where $\bs X_i^{(m)}:=\Pi_m\bs X_i,i=1,...,n$. }
In accordance with the truncation on the predictor $\X$, we turn to estimate the \textit{truncated central subspace} $\mc S_{\Y|\bs X}^{(m)}$:
\begin{align}\label{def: truncated central subspace}\mc S_{\Y|\bs X}^{(m)}=\Pi_m \mc S_{\Y|\bs X}=\mathrm{span}\{\bs{\beta}_1^{(m)},\dots,\bs{\beta}_d^{(m)}\},\end{align}
 where $\bs{\beta}_{k}^{(m)}:=\Pi_m(\bs{\beta}_k)~\text{for}~k=1,\dots,d$.
To estimate $\mc S_{\Y|\bs X}^{(m)}$, we need a fundamental lemma.
\begin{lemma}\label{lemma, way of estimate truncate central subspace}
Under Assumption $\ref{as:joint distribution assumption}$, we have:
\begin{align*}
\mathcal{S}^{(m)}_{{\Y|\boldsymbol{X}}}=\Gamma_m^\dagger\mathrm{Im}\{\mathrm{MDDO}(\boldsymbol{X}^{(m)}|\Y)\}.
\end{align*}
\end{lemma}
Consequently, it amounts to estimating $\Gamma_m^\dagger\mathrm{Im}\{\mathrm{MDDO}(\boldsymbol{X}^{(m)}|\Y)\}$. To this end, we estimate both $\Gamma_m^\dagger$ and $\mathrm{Im}\{\mathrm{MDDO}(\boldsymbol{X}^{(m)}|\Y)\}$. 
From now on, we abbreviate
$\mathrm{MDDO}(\boldsymbol{X}|\Y)$, $\mathrm{MDDO}(\boldsymbol{X}^{(m)}|\Y)$ and $\widehat{\mathrm{MDDO}}(\boldsymbol{X}^{(m)}|\Y)$ to $M$, $M_m$ and $\widehat M_m$ respectively. 

We are in a position to state the estimation procedures of $\mathcal{S}^{(m)}_{{\Y|\boldsymbol{X}}}$.
Define the estimator of $\Gamma_m$ and $M_m$ as follows:
\begin{align*}
\widehat{\Gamma}_m:=\frac{1}{n}\sum_{i=1}^n \bs X_i^{(m)}\otimes \bs X_i^{(m)};\quad\wh M_m:=- \frac{1}{n^2}\sum_{j,k=1}^n\bs X_j^{(m)}\otimes \bs X_k^{(m)}\|\Y_j-\Y_k\|
\end{align*}
where $\bs X_i^{(m)}:=\Pi_m\bs X_i,i=1,...,n$.
Then the estimator of $\mc S_{\Y|\bs X}^{(m)}$ can be defined by
\begin{align}\label{eq:def of hat Sm}
\wh{\mc S}_{\Y|\bs X}^{(m)}:=\widehat{\Gamma}_m^\dagger\mathrm{Im}^d\lb\widehat M_m\rb, 
\end{align}
where $\mathrm{Im}^d\lb\widehat M_m\rb$ denotes the space spanned by the top $d$ eigenfunctions of $\wh M_m$.

An equivalent definition of $\wh{\mc S}_{Y|\bs X}^{(m)}$ can be derived as follows: define $\wh M_m^d$ by:
\begin{align}\label{wh M_m spectral decomposition}\wh M_m=\sum_{i=1}^\infty\wh\mu_i\wh\gamma_i\otimes\wh\gamma_i\quad\text{and}\quad \wh M_m^d:=\sum_{i=1}^d\wh\mu_i\wh\gamma_i\otimes\wh\gamma_i,\end{align}
where $\lb\wh\mu_i\rb_{i\in\Z_{\geq1}}$ are eigenvalues with descending order of $\wh M_m$, and $\wh\gamma_i$'s are eigenfunctions of $\wh M_m$ associated with $\wh\mu_i$. 
Then
\begin{align}\label{def: estimator central subspace}
\wh{\mc S}_{\bs{Y}|\bs X}^{(m)}= \mathrm{Im}\lb\widehat{\Gamma}_m^\dagger\widehat M_m^d\rb.\end{align}

Based on \eqref{eq:def of hat Sm}, we introduce detailed FSFIR procedures in Algorithm $\ref{alg:slicing-free}$.
\begin{algorithm}[H]
\setstretch{1.5}
\begin{algorithmic}
\caption {FSFIR. }\label{alg:slicing-free} 
\State 
\begin{enumerate}
\item Standardize $\{\X_i,1\leqslant i\leqslant n\}$, i.e., $\bs Z_i=\bs X_i-n^{- 1}\sum_{i=1}^n \bs X_{i};$
\item Do truncation: 
% for any $\gamma\in\l0,\tfrac{5}{4(\alpha_1+\alpha_2+3)}\r$, choosing $m\propto n^{\frac{1-2\gamma}{2\alpha_1+2\alpha_2+1}}$ 
choose some $m$ and then obtain
 $\bs Z^{(m)}_{i}=\Pi_{m}\bs Z_{i}$;
\item Form the estimator $\widehat M_m$ and $\widehat\Gamma_m$ according to 
 \begin{equation*}
\widehat M_m=-\frac{1}{n^{2}}\sum_{j,k=1}^n\bs Z_j^{(m)}\otimes \bs Z_k^{(m)}\|\Y_j-\Y_k\|\quad\text{and}\quad \widehat\Gamma_m=\frac1n\sum_{i=1}^n \bs Z_i^{(m)} \otimes \bs Z_i^{(m)}
 \end{equation*}
 respectively;
\item Find the top $d$ eigenfunctions of $ \widehat{M}_{m}$ and denote them by $\widehat{\gamma}_{k} (k=1,\ldots,d)$;
\item Figure out $\widehat{\bs{\beta}}_{k}=\widehat{\Gamma}^{\dagger}_{m}\widehat{\gamma}_{k}(k=1,\dots,d$);
\end{enumerate}
\State  Return $\widehat{\mc S}_{Y|\bs X}^{(m)}=\mathrm{span}\lb\widehat{\bs{\beta}}_1,...,\widehat{\bs{\beta}}_d\rb$.
\end{algorithmic}
\end{algorithm}

% \begin{breakablealgorithm}
% \caption {FSFIR. ($\alpha_1,\alpha_2$ as in Assumption $\ref{assumption: rate-type condition}$)}
% \label{alg:slicing-free} 
% \noindent Standardize $\{\X_i,1\leqslant i\leqslant n\}$,i.e.,  $\bs Z_i=\bs X_i-n^{- 1}\sum_{i=1}^n \bs X_{i};$
% \begin{enumerate}
% \item Do truncation: for any $\gamma\in\l0,\tfrac{5}{4(\alpha_1+\alpha_2+3)}\r$, choosing $m\propto n^{\frac{1-2\gamma}{2\alpha_1+2\alpha_2+1}}$ 
%  and then obtain
%  $\bs Z^{(m)}_{i}=\Pi_{m}\bs Z_{i}$;
% \item Form the estimator $\widehat M_m$ and $\widehat\Gamma_m$ according to 
%  \begin{equation*}
% \widehat M_m=-\frac{1}{n^{2}}\sum_{j,k=1}^n\bs Z_j^{(m)}\otimes \bs Z_k^{(m)}|Y_j-Y_k|\quad\text{and}\quad \widehat\Gamma_m=\frac1n\sum_{i=1}^n \bs Z_i^{(m)} \otimes \bs Z_i^{(m)}
%  \end{equation*}
%  respectively;
% \item Find the top $d$ eigenfunctions of $ \widehat{M}_{m}$ and denote them by $\widehat\eta_{k} (k=1,\ldots,d)$;
% \item Figure out $\widehat{\bs{\beta}}_{k}=\widehat{\Gamma}^{\dagger}_{m}\widehat\eta_{k}(k=1,\dots,d$);
% \end{enumerate}
% \noindent Return $\widehat{\mc S}_{Y|\bs X}^{(m)}=\mathrm{span}\lb\widehat\bs{\beta}_1,...,\widehat\bs{\beta}_d\rb$.
% \end{breakablealgorithm}

Our FSFIR algorithm provides an optimal selection criterion for $m$, 
precisely $m\propto n^{\frac{1-2\gamma}{2\alpha_1+2\alpha_2+1}}$  for arbitrary $\gamma\in\l0,\tfrac{5}{4(\alpha_1+\alpha_2+3)}\r$, $\alpha_1$ and $\alpha_2$ are defined in Assumption $\ref{assumption: rate-type condition}$.
In practice, it should be better to choose $m=tn^{\frac{1-2\gamma}{2\alpha_1+2\alpha_2+1}}$ for some $t\in[1/\log(n),\log(n)]$ which may be determined by cross-validation.

\begin{remark}
Note that the estimation method in Algorithm $\ref{alg:slicing-free}$ can be realized without specifying any slice number $H$. For the effectiveness of FSFIR, see the next Section $\ref{subsection, convergence rate}$.
\end{remark}

\subsection{Convergence rate of FSFIR estimator}\label{subsection, convergence rate}
% Now that we have established the estimator of the $\widehat{\Gamma}_m^\dagger\widehat M_m$ in order to estimate the central subspace, 
 Before stating our main result, we need a uniform sub-Gaussian assumption.
\begin{assumption}\label{assumption: sub-Gaussian}
There exist two positive constants $\sigma_0$ and $\sigma_1$ such that\begin{equation}\label{eq:uniform subgaussian bound}
 \sup_m\max_{1\leqslant j\leqslant m}\mathbb E\left[\exp\big(2\sigma_0\langle \boldsymbol{X},\phi_j\rangle^2\big)\right]\leqslant \sigma_1\quad\text{and} \quad \mathbb E\left[\exp\big(2\sigma_0\|\Y\|^2\big)\right]\leqslant \sigma_1. 
 \end{equation}
 \end{assumption}
These inequalities  generalize \cite[Condition (C1)]{mai2021slicing}.
Similar sub-Gaussian type conditions are commonly used in SIR literature \citep{lin2018consistency,lin2019sparse,lin2021optimality,huang2023sliced}.
We now derive a large deviation inequality between $\widehat M_m^d$ and $M_m$.
\begin{proposition}\label{prop:bound hatMmd Mm}
Under Assumptions $\ref{as:joint distribution assumption}$ and $\ref{assumption: sub-Gaussian}$, for all $\gamma\in(0,1/2)$, there exist positive constants $D_0=D_0(\gamma,\sigma_0,\sigma_1)$, $D_1=D_1(\sigma_1)$, $D_2=D_2(\sigma_0,\sigma_1)$ and $n_0=n_0(\gamma,\sigma_0,\sigma_1)$ such that for all $n\geqslant n_0$ and
$$C\in \l D_0n^{\frac{2\gamma}{5}}-\ln\l D_1m^2n \r,D_2 n^{\frac{1}{5}}-\ln\l D_1m^2n \r \rmi,$$ 
we have
\begin{equation*}
\mathbb{P}\l\|\wh M_m^d- M_m\| <\l \frac{C+\ln( D_1m^2n)}{D_2}\r^{\frac52}\frac{24m}{\sqrt n}\r\geqslant 1-\exp(- C).
\end{equation*}
\end{proposition}

Thanks to this proposition, we can derive a concentration inequality for $\wh\Gamma_m^\dagger\wh M_m^d$ around its population counterpart $\Gamma_m^\dagger M_m$. Before we get a hand on this, we recall the following rate-type condition which is fundamental in functional data analysis \citep{ferre2003functional, ferre2005smoothed,Hall2007mcflr,lei2014adaptive,lian2015functional,chen2023optimality}. 
\begin{assumption}[Rate-type condition]\label{assumption: rate-type condition}
There exist positive constants $\alpha_1$, $\alpha_2$, $\wt C$ and $\wt C'$ such that
\begin{align*}
\alpha_2>1/2,\quad\lambda_j\geqslant \wt Cj^{- \alpha_1}\quad\text{and}\quad |b_{ij}|\leqslant \wt C'j^{- \alpha_2}, \quad(\forall i\in[d],j\in\mb Z_{+})
\end{align*}
 where $b_{ij}:=\langle  \bs{\beta}_i,\phi_j\rangle$ for $\bs\beta_i$ defined in \eqref{def: central subspace}.
\end{assumption}
\begin{remark}
 The assumption about the eigenvalues $\lambda_{j}$ of $\Gamma$ ensures that the estimation of eigenfunctions of $\Gamma$ is accurate. The assumption about the coefficients $b_{ij}$ implies that they do not decrease too quickly concerning $j$ uniformly for all $i$. This assumption also implies that any basis $\{\wt\bbeta_i\}_{i=1}^d$ of $\mc S_{\Y|\vX}$ that satisfies $\wt\bbeta_i=\sum_{j=1}^\infty\wt{b}_{ij}\phi_j$ has coefficients $|\wt b_{ij}|\lesssim j^{-\alpha_2}$.
% The decay rate type assumptions such as Assumption $\ref{assumption: rate-type condition}$ or its variant are generally required when studying the coverage rate of FSIR \citep{ferre2003functional, ferre2005smoothed,Hall2007mcflr,lei2014adaptive,lian2015functional,chen2023optimality}.
 \end{remark}
\begin{proposition}
\label{prop:concentration Gammam dag Mmd}
 Suppose that Assumptions $\ref{as:joint distribution assumption}$ to $\ref{assumption: rate-type condition}$ hold, then $\forall \gamma\in(0,1/2)$, we have
\begin{equation*}
\begin{aligned}
 \lno\widehat\Gamma_m^\dagger \widehat M_m^d-\Gamma_m^\dagger M_m\rno =O_{\mb{P}}\l \frac{m^{\alpha_1+1}}{n^{1/2-\gamma}} \r.
\end{aligned}
\end{equation*}
\end{proposition}
This proposition is used to give an error bound of the subspace estimation error in the subsequent Theorem $\ref{theorem, total convergence rate}$.
Now we state the convergence rate of our FSFIR estimator:
\begin{theorem}\label{theorem, total convergence rate}
Assume Assumptions $\ref{as:joint distribution assumption}$ to $\ref{assumption: rate-type condition}$ hold. Then for any $\gamma\in\l0,\tfrac{5}{4(\alpha_1+\alpha_2+3)}\r$, by choosing 
$m=n^{\frac{1-2\gamma}{2\alpha_1+2\alpha_2+1}}$ (i.e.,  $c_1=\frac{1-2\gamma}{2\alpha_1+2\alpha_2+1}$), we can get that
% \begin{align*}
%  \left\|P_{\mc{S}_{Y|\X}}-P_{\widehat{ \mc{S}}_{Y|\X}^{(m)}}\right\| =O_{\mb{P}}\l n^{-\frac{(2\alpha_2-1)(1-2\gamma)}{2(2\alpha_1+2\alpha_2+1)}}\r. 
% \end{align*}
\begin{align*}
 \mb E\left[\left\|P_{\mc{S}_{\Y|\X}}-P_{\widehat{ \mc{S}}_{\Y|\X}^{(m)}}\right\|^2\right] \lesssim n^{-\frac{(2\alpha_2-1)(1-2\gamma)}{2\alpha_1+2\alpha_2+1}}. 
\end{align*}
\end{theorem}

This specific convergence rate guarantees the effectiveness of FSFIR. To prove this convergence rate, we decompose the error into two parts: $\mathbf{Loss}_1$  caused by truncation which is easy to bound and $\mathbf{Loss}_2$ caused by estimating $\wh{\mc S}_{Y|\bs X}^{(m)}$ with finite samples. Our main job is to bound the latter one. To this end, we apply the generalized  Sin Theta theorem in \cite[Proposition 2.3]{seelmann2014notes} to non-symmetric operator. Then, $\mathbf{Loss}_2$ is bounded by 
combining this non-symmetric Sin Theta theorem with Proposition $\ref{prop:concentration Gammam dag Mmd}$.

\section{Numerical Performance of FSFIR}\label{section, experiments}

In this section, we study the numerical performance of FSFIR from several aspects. The first  experiments focuses on the empirical subspace estimation error performance of FSFIR for estimating the central subspace in  synthetic data.
The  experiment includes the comparison with some well-known FSIR methods  including the truncated FSIR \citep{ferre2003functional,chen2023optimality} and regularized FSIR \citep{lian2015functional}. The results reveal the  advantage and  convenience of FSFIR to practice.
Then we apply FSFIR algorithm to a real data: the bike sharing data to  demonstrate the efficiency of our algorithm.
\subsection{Synthetic experiments}\label{sec:Synthetic}
We first introduce the synthetic models we considered in this subsection. All synthetic models are of a functional-valued predictor.
Throughout this section, we set  $\varepsilon\sim N(0,0.25)$, i.e., a noise level of $0.25$. The  experimental results  corresponding to a higher noise level such as $1$ are presented in the Supplementary Materials.
\begin{example}\citep[Example 1]{lei2014adaptive,lee2020testing} 
First, let $\bar\beta_1=0.3$, $\bar\beta_j=4(- 1)^jj^{- 2}(j\geqslant2)$ and $\beta_j=\bar\beta_j/\|\bar\beta\|$ where $\bar\beta=\{\bar\beta_n\}_{n=1}^\infty$. Then we define $\bs{\beta}(t)=\sum\limits_{j=1}^{100}\beta_j\phi_j(t)$ where $\phi_1(t)=1$ together with $\phi_j(t)=\sqrt2 \cos[(j-1)\pi t](j\geqslant2)$ form an orthonormal basis. Then we define model $\mc M_1$ as follows:
\begin{align*}
\mc M_1:
&\quad~Y=\langle \bs X(t),\bs{\beta}(t)\rangle+\varepsilon;\\
&\bs X(t)=\sum_{j=1}^{100} j^{- 0.55}X_j\phi_j(t),t\in[0,1],
\end{align*}
where $X_j\overset{\mathrm{iid}}{\sim}N(0,1)$.
\end{example}

\begin{example}\citep[example M1]{lian2015functional}
Consider $\bs{\beta}_1(t)=\sqrt2\sin\left(\frac32\pi t\right)$ and $\bs{\beta}_2(t)=\sqrt2\sin\left(\frac52\pi t\right)$. Define model $\mc M_2$ as follows:
\[
\mc M_2: Y=\langle\bs{\beta}_1,\bs X\rangle+100\langle\bs{\beta}_2,\bs X\rangle^3+\varepsilon,
\]
where $\bs X$ is a standard Brown motion and we approximate it by the top 100 eigenfunctions of the Karhunen–Loève decomposition in practical implementation.
\end{example}

\begin{example}
    $Y=\exp(\langle \bs \beta,\boldsymbol{X}\rangle )+\epsilon$,
    where $\boldsymbol{X}$ is the standard Brownian motion on $[0,1]$, and $\bs\beta=\sqrt2\sin(\frac{3\pi t}{2})$.
\end{example}

% \begin{example}
% Let
% \begin{equation*}
% \beta_{1,j}=\left\{\begin{matrix}
% 8j^{- 2}, & j~\mathrm{is~even};\\
%  0 & j~\mathrm{is~odd},
% \end{matrix}\right.
% \qquad
% \beta_{2,j}=\left\{\begin{matrix}
% 8j^{- 2}, & j~ \mathrm{is~ odd};\\
%  0, & j~ \mathrm{is~ even}, 
% \end{matrix}\right.
% \end{equation*}
% and define:
% \begin{equation*}
% \beta_1(t)=\sum_{j=1}^{100}\beta_{1,j}\phi_j(t);\quad\beta_2(t)=\sum_{j=1}^{100}\beta_{2,j}\phi_j(t),
% \end{equation*}
% where $\phi_n(t):=\sqrt2\sin[(n-\frac12)\pi t](n\geqslant 1)$ is an orthonormal basis. Define model $\mc M_3$ as follows:
% \[
% \mc M_3:{Y}=100\langle\beta_1(t),\bs X(t)\rangle^3+\langle\beta_2(t),\bs X(t)\rangle+\varepsilon
% \]
% where $\bs X$ is also a standard Brown motion.
% \end{example}

In the following, we compare our FSFIR method with several slice-based methods using models $\mc M_1$ to $\mc M_3$.
Consider two slice-based methods --- one is truncated FSIR \citep{ferre2003functional,chen2023optimality}, which studies a truncation on the covariance operator, and the other is regularized FSIR \citep{lian2015functional}, which estimates $\S$ by applying a regularization tune parameter $\rho$ on $\Gamma$. In the following, we abbreviate these two methods as TFSIR and RFSIR respectively. 
% We abbreviate these two slice-based methods to \textit{truncated FSIR} and \textit{regularized FSIR} respectively.

In this experiment, we set the sample size $n=20000$. For slice-based methods, we set the slice number $H=10$, a popularly adopted slice number. To evaluate the performance of these methods, we choose the  subspace estimation error: $\mc D(\bs B;\bs{\wh{B}}):=\left\|P_{\bs B}-P_{\bs{\wh B}}\right\| $ where $\bs B:=(\bs\beta_1,...,\bs\beta_d):\R^d\to \mc H$, $\bs{\wh B}:=(\wh{\bs\beta}_1,...,\wh{\bs\beta}_d):\R^d\to \mc H$.
This metric takes value in $[0,1]$ and the smaller it is, the better the performance.

For each model, we choose  several $m$'s for FSFIR and TFSIR, among which one tends to have the best performance. 
% Following the same fashion, we select several $\rho$'s for RFSIR.
Specifically, 
$m$ ranges in $\{2,3,\dots,13,14,20,30,40\}$. Following the same fashion,  $\rho$ ranges in $0.01\times \{1,2,\cdots,9,10,15,20,25,30,40,\cdots,140,150\}$. Each trial is repeated 100 times for reliability.
We show the average $\mc D(\bs B;\bs{\wh B})$ with different $m$ or $\rho$ for three methods under $\mc M_1$ to $\mc M_3$ in Figure \ref{fig:error 3models},
where we mark minimal error in each model with red `$\times$'. The shaded areas represent the standard error associated with these estimates and all of them are less than  $0.009$. For FSFIR, the  minimal errors for $\mc M_1-\mc M_3$ are  $0.06,0.01,0.01$ respectively.
For TFSIR, the  minimal errors are  $0.06,0.02,0.01$ and for regularized FSIR,  the  minimal errors are $0.09,0.04,0.01$.  

\begin{figure}[H]% [H] is so declass\'e!
	\centering
	\begin{minipage}{0.33\textwidth}
		\includegraphics[width=\textwidth]{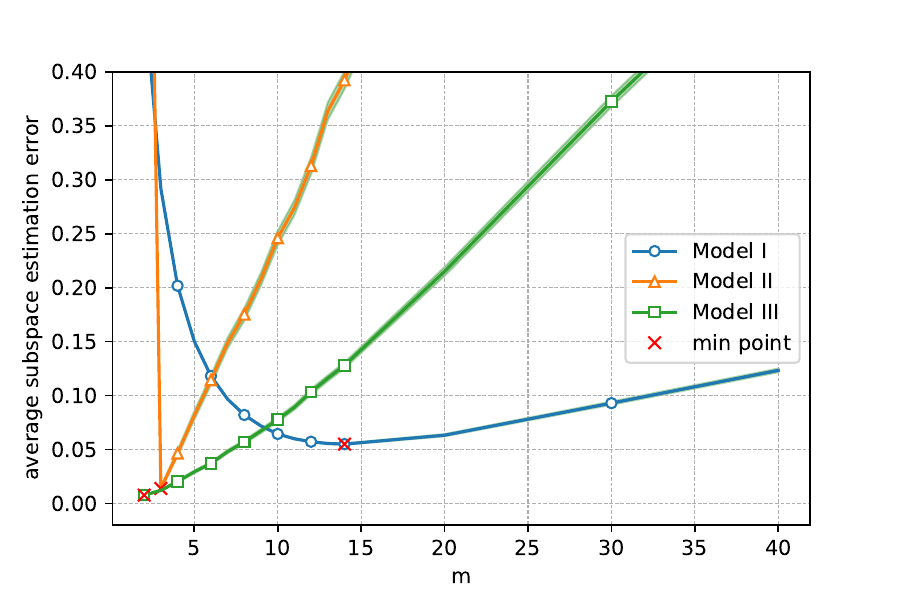}
	\end{minipage}\hfill
	\begin{minipage}{0.33\textwidth}
		\includegraphics[width=\textwidth]{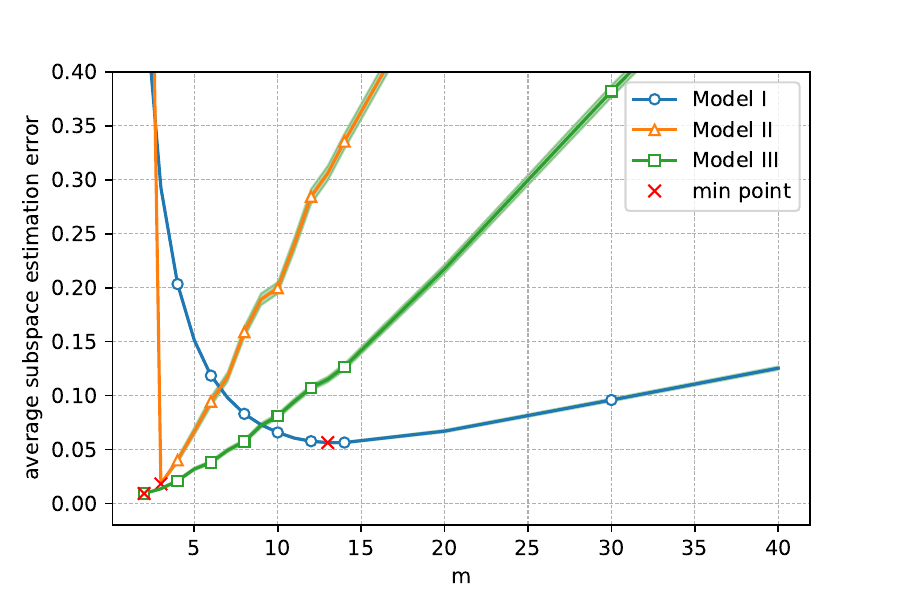}
	\end{minipage}
	\begin{minipage}{0.33\textwidth}
		\includegraphics[width=\textwidth]{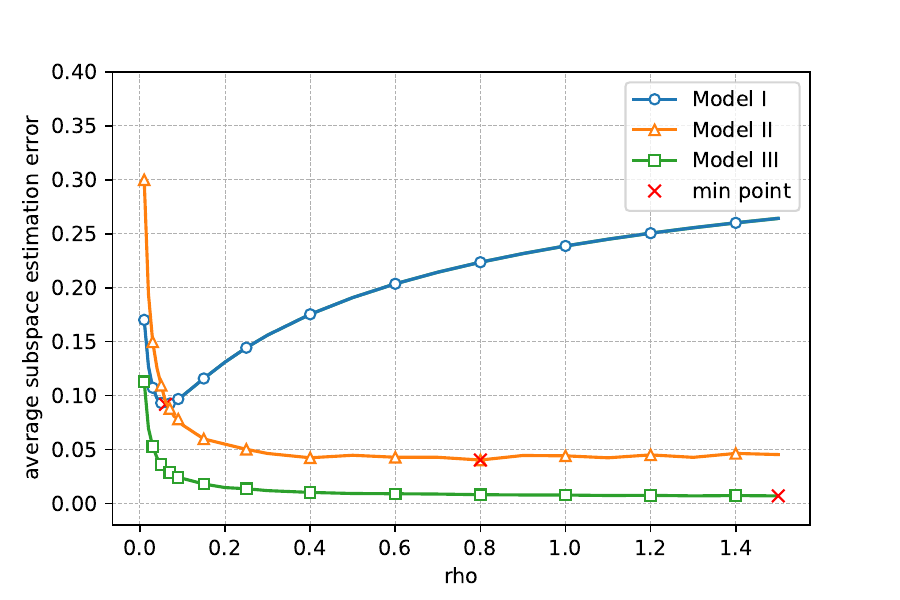}
	\end{minipage}
\caption{Average subspace estimation error of FSFIR (left), TFSIR (middle) and RFSIR (right) for various models. The standard errors are all below $0.009$. }
\label{fig:error 3models}
\end{figure}
% and \textit{standard error} (SE) multiplied by 1000 in Table $\ref{table, Comparsion of different methods}$.

Figure \ref{fig:error 3models} shows that FSFIR attains the best performance among  all models. 
Moreover, FSFIR is easier to practice as it does not need a slice number $H$ in advance. 
\subsection{Application to real data}
In this section, we analyze the bike sharing data \citep{fanaee2014event,lee2022functional}, which includes hourly bike rental counts and weather information such as temperature, precipitation, wind speed, and humidity. The data was collected every day from January 1, 2011 to December 31, 2012 from the Capital Bike Share system in Washington, DC and can be found at \url{https://archive.ics.uci.edu/ml/datasets/Bike+Sharing+Dataset}.

The main goal of this section is to investigate how temperature affects bike rentals on Saturdays. After removing data from three Saturdays with significant missing information, we plot hourly bike rental counts and hourly normalized temperature (values divided by 41, the maximum value) for 102 Saturdays in Figure \ref{figure, comparsion of MDDO and FSIR, real data}. In the following analysis, we use hourly normalized temperature and the logarithm of the daily average  bike rental counts as the predictor function and scalar response, respectively.
\begin{figure}[H]% [H] is so declass\'e!
	\centering
	\begin{minipage}{0.50\textwidth}
		\includegraphics[width=\textwidth]{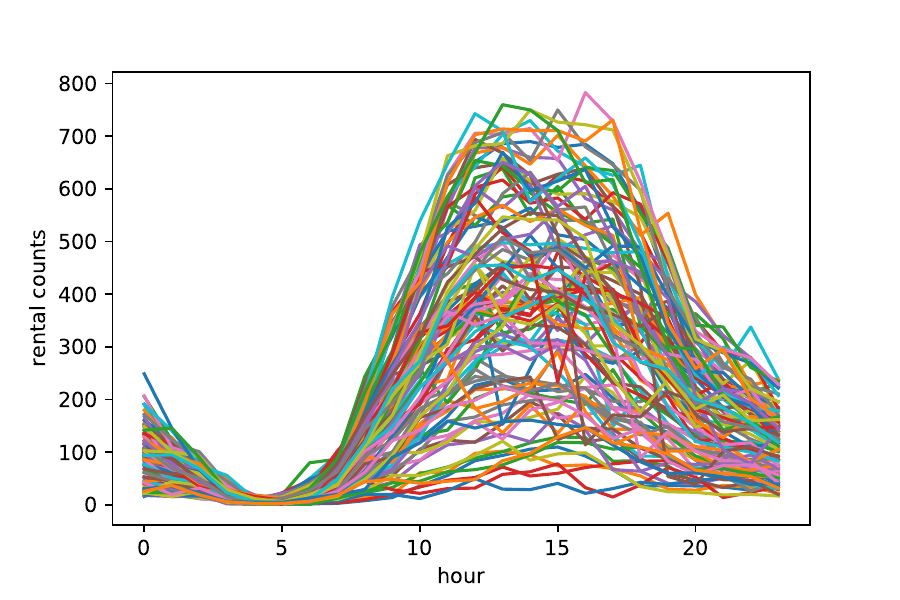}
		%\label{fig:coh1}
	\end{minipage}\hfill
	\begin{minipage}{0.50\textwidth}
		\includegraphics[width=\textwidth]{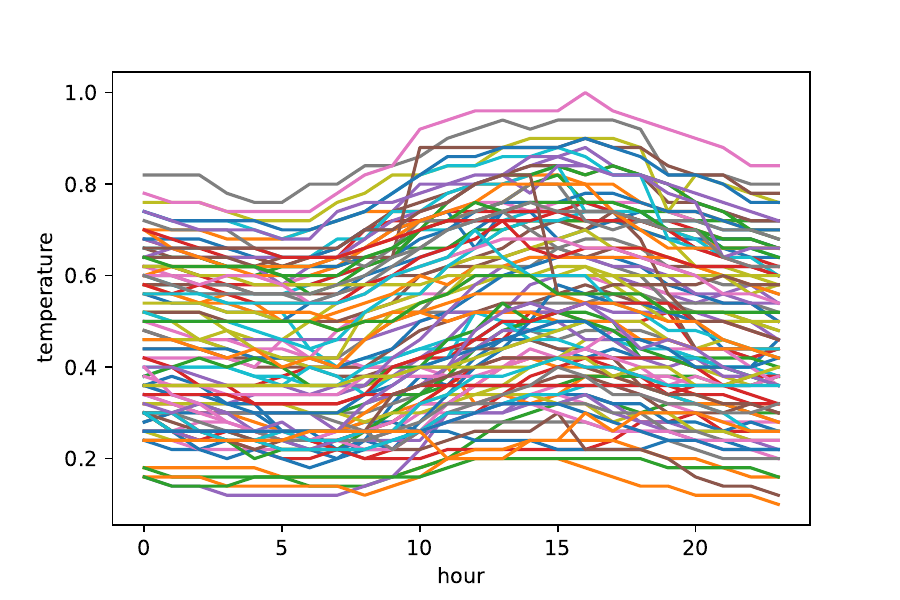}
	\end{minipage}
 \caption{Bike sharing data}
\label{figure, comparsion of MDDO and FSIR, real data}
\end{figure}
% After reducing the dimension using FSFIR, we used Gaussian process regression to fit a nonparametric regression model. 

To evaluate the estimation error performance of  FSFIR for estimating the central space, we incorporate dimension reduction with FSFIR as an intermediate step in modeling the relationship between the predictor and response variables. Specifically, we apply FSFIR  to perform dimension reduction on a given training dataset $\{(\vX_i,Y_i)\}_{i=1}^n$. This yields a set of low-dimensional predictors $\bs x_i$ for each $i\in [n]$. Subsequently, we utilize Gaussian process regression to fit a nonparametric regression model using the samples $\{(\bs x_i,Y_i)\}_{i=1}^n$.
We randomly selected $90$ samples as the training data and used the remaining data to calculate the out-of-sample mean square error (MSE). We repeated this process $100$ times and calculated the mean and standard error. The results are presented in Table \ref{tab:real data error},  which suggests that FSFIR performs well in practical applications. 
It is noteworthy that the best
result of  FSFIR  is observed when $d=1$. 
This means  FSFIR  provides an  accurate and simpler (lower dimensional) model for the relationship between the response variable and the predictor. 
% As a result, FSFIR is highly effective in detecting the most significant associations between the response variable and the predictor, making it an efficient method for dimension reduction.
\begin{table}[H]
\renewcommand\arraystretch{0.3}
\begin{tabular}{|c|c|c|c|c|c|c|c|c|}
\hline
 & $m$   & $1$ & $3$ & $5$ & $7$ & $9$  & $11$ & $13$  \\
\hline
\hline
\multirow{5}{*}{FSFIR} 
& $d=1$  & \emph{\textbf{0.230}} & \textbf{0.259} & \textbf{0.265} & \textbf{0.247} & \textbf{ 0.280} & \textbf{ 0.319}  & \textbf{0.320} \\ 
& & {\footnotesize(0.0097)} & {\footnotesize(0.0126)} & {\footnotesize(0.0127)} & {\footnotesize(0.0122)} & {\footnotesize(0.0128)} & {\footnotesize(0.0143)} & {\footnotesize(0.0160)} 
\\
& $d=2$  &  & {\textbf{0.356}} & \textbf{0.276} & \textbf{0.372} & \textbf{0.334} & \textbf{ 0.396} & \textbf{0.329} \\ 
& & & {\footnotesize(0.0409)} & {\footnotesize(0.0170)} & {\footnotesize(0.0460)} & {\footnotesize(0.0312)} & {\footnotesize(0.0226)}  & {\footnotesize(0.0170)} 
\\ 
& $d=3$  &  & \textbf{0.358} & \textbf{0.461} & \textbf{0.370} & \textbf{0.420} & \textbf{ 0.625}& \textbf{0.396} \\ 
& &  & {\footnotesize(0.0303)} & {\footnotesize(0.0507)} & {\footnotesize(0.0304)} & {\footnotesize(0.0390)} & {\footnotesize(0.0795)}  & {\footnotesize(0.0365)} 
\\
& $d=4$  &  &  & \textbf{0.699} & \textbf{0.473} & \textbf{0.726} & \textbf{ 0.831}& \textbf{0.460} \\ 
& &  &  & {\footnotesize(0.0544)} & {\footnotesize(0.0584)} & {\footnotesize(0.0854)} & {\footnotesize(0.1008)}   & {\footnotesize(0.0374)} 
\\
& $d=5$  & &  & \textbf{1.052} & \textbf{0.876} & \textbf{1.131} & \textbf{0.883} & \textbf{0.936}\\ 
& &  &  & {\footnotesize(0.0710)} & {\footnotesize(0.0682)} & {\footnotesize(0.0930)} & {\footnotesize(0.0942)}    & {\footnotesize(0.0875)} 
\\[1pt]
\hline
\end{tabular}
\caption{The empirical mean (standard error) of MSE.}\label{tab:real data error}
\end{table}

\section{Concluding Remarks}\label{section, Discussion}
In summary, we introduce two novel objects, the statistics MDDO and the method FSFIR.
MDDO serves to measure the conditional mean independence of a functional-valued predictor on a multivariate response. And based on MDDO, FSFIR aims to estimating the central subspace $\mc S_{\Y|\vX}$.

Besides MDDO, there are other ways to examine the conditional mean
independence in functional-data cases, such as the \textit{functional martingale
difference divergence} (FMDD, \citealt{lee2020testing}). However, we would like to point out an extra feature of our MDDO --- Under certain circumstances, a low rank projection of $\bs X$ is conditional mean independent of $Y$ even if $\X$ is not. In other words, for some $e\in\mc H$, $\E[\langle e,\X\rangle|Y]=\E[\langle e,\X\rangle]$ but $\E[\X|Y]$ is not equal to $\E[\X]$. 
In this case, we can use MDDO to
separate out a part that is conditional mean independent of $Y$ in view of Theorem $\ref{theorem, MDDO and conditional mean independence}$ (ii). This property of MDDO makes it a tool for slicing-free estimation (see the proof of Theorem $\ref{theorem, MDDO and IRS}$).

However, there are still several open problems that remain to be solved. For example, a study of FSFIR   from a decision theoretic point of view  is interesting. Additionally, it would be interesting to extend these methods to cases with high-dimensional  or functional-valued response. We plan to explore these topics in future research.

\section*{Supplementary Materials}
Supplement to ``Functional Slicing-free Inverse Regression via Martingale Difference Divergence Operator''. The supplementary material includes
the proofs for all the theoretical results in the paper.

\bibliography{reference}
\bibliographystyle{chicago}

\bigskip
\vskip .65cm

\noindent
Songtao Tian, Department of Mathematical Sciences, Tsinghua University
\vskip 2pt
\noindent
E-mail: tst20@mails.tsinghua.edu.cn

\noindent
Zixiong Yu, Department of Mathematical Sciences, Tsinghua University
\vskip 2pt
\noindent
E-mail: yuzx19@mails.tsinghua.edu.cn

\noindent
Rui Chen, Center for Statistical Science, Department of Industrial Engineering, Tsinghua University
\vskip 2pt
\noindent
E-mail: chenrui\_fzu@163.com
\vskip 2pt

\input{supplement.tex}

\end{document}

%% file: supplement.tex
\newpage
\appendix

\section{Proof of Lemma \ref{lemma, equivalence of two def of MDDO}}
\begin{proof}
For any ${\bs{\beta}}\in\mc H$, according to the definition of $G_{\bs s}$ (see Definition $\ref{def: MDDO}$), one has
\begin{align*}
\langle G_{\bs s},{\bs{\beta}}\rangle&=\int_{[0,1]} G_{\bs s}(t){\bs{\beta}}(t)~\mathrm{d}t=\int_{[0,1]}\mathrm{cov}\hspace{-0.9mm}\left(\bs{X}(t),\mathrm{e}^{\mi\langle \bs s,\Y\rangle}\right){\bs{\beta}}(t)~\mathrm{d}t\\
&=\int_{[0,1]}\mathrm{cov}\hspace{-0.9mm}\left(\bs{X}(t){\bs{\beta}}(t),\mathrm{e}^{\mi \langle \bs s,\Y\rangle}\right)~\mathrm{d}t.
\end{align*}
By Fubini theorem, under Assumption $\ref{as:joint distribution assumption}$, one can exchange the order of integration and covariance above and get that
\begin{align*}
 \langle G_{\bs s},{\bs{\beta}}\rangle&=\int_{[0,1]}\mathrm{cov}\hspace{-0.9mm}\left(\bs{X}(t){\bs{\beta}}(t),\mathrm{e}^{\mi \langle \bs s,\Y\rangle}\right)~\mathrm{d}t\\ &=\mathrm{cov}\hspace{-0.9mm}\left(\int_{[0,1]}\bs{X}(t){\bs{\beta}}(t)~\mathrm{d}t,\mathrm{e}^{\mi \langle \bs s,\Y\rangle}\right)=\mathrm{cov}\hspace{-0.9mm}\left(\langle \bs{X},{\bs{\beta}}\rangle,\mathrm{e}^{\mi \langle \bs s ,\Y\rangle}\right).
\end{align*}
Thus for any $\bs\alpha(t),{\bs{\beta}}(t)\in\mc H$, one can get
\begin{align*}
\big\langle \big(G_{\bs s}\otimes \overline{G}_{\bs s}\big)\bs\alpha,{\bs{\beta}}\big\rangle=\langle G_{\bs s},\bs\alpha\rangle\langle \overline{G}_{\bs s},{\bs{\beta}}\rangle=\mathrm{cov}\hspace{-0.9mm}\left(\langle \bs{X},\bs\alpha\rangle,\mathrm{e}^{\mi \langle \bs s,\Y\rangle}\right)\hspace{-0.9mm}\mathrm{cov}\hspace{-0.9mm}\left(\langle \bs{X},{\bs{\beta}}\rangle,\mathrm{e}^{-\mi\langle \bs s,\Y\rangle}\right)\\
=\mb{E}\hspace{-0.9mm}\left(\langle \bs{X},\bs\alpha\rangle\mathrm{e}^{\mi \langle \bs s,\Y\rangle}\right)\mb{E}\hspace{-0.8mm}\left(\langle \bs{X},{\bs{\beta}}\rangle\mathrm{e}^{-\mi \langle \bs s,\Y\rangle}\right)=\mb{E}\Big(\langle \bs{X},\bs\alpha\rangle\langle \bs{X}',{\bs{\beta}}\rangle\mathrm{e}^{\mi \langle \bs s,\Y-\Y'\rangle}\Big).
\end{align*}
Considering that $\mb{E}\big(\langle \bs{X},\alpha\rangle\langle \bs{X}',{\bs{\beta}}\rangle\big)=0$, one has
\begin{align*}
\big\langle \big(G_{\bs s}\otimes \overline{G}_{\bs s}\big)\bs\alpha,{\bs{\beta}}\big\rangle
=- \mb{E}\Big(\langle \bs{X},\bs\alpha\rangle\langle \bs{X}',{\bs{\beta}}\rangle\big(1-\mr{e}^{\mi \langle \bs s,\Y-\Y'\rangle}\big)\Big)&\\
=- \mb{E}\Big(\langle \bs{X},\bs\alpha\rangle\langle \bs{X}',{\bs{\beta}}\rangle\big[1-\cos\big(\langle \bs s,\Y-\Y'\rangle\big)\big]\Big)&\\
+\mi\mb{E}\Big(\langle \bs{X},\bs\alpha\rangle\langle \bs{X}',{\bs{\beta}}\rangle\big[\sin\big(\langle\bs s,\Y-\Y'\rangle\big)\big]\Big)&.
\end{align*}
It is easy to check that
\[\int_{\mb R^q}\frac{\sin \big(\langle\bs s,\Y-\Y'\rangle)\big)}{\|\bs s\|^{1+q}}~\mr{d}\bs s=\lim_{\varepsilon\to0^+}\int_{\bs s\in\mb{R}^q:\varepsilon\leqslant\|\bs s\|\leqslant \varepsilon^{-1}}\frac{\sin \big(\langle \bs s,\Y-\Y'\rangle\big)}{\|\bs s\|^{1+q}}~\mr{d}\bs s=0,\]
because the integrand is an odd function. By Lemma 1 in \cite{szekely2007measuring},  one can also get
\[\int_{\R^q}\frac{1-\cos\big(\langle \bs s,\Y-\Y'\rangle\big)}{\|\bs s\|^{1+q}}~\mr{d}\bs s=c_q\|\Y-\Y'\|.
\]
Combining above results with Definition $\ref{def: MDDO}$, one can obtain that 
\begin{align}\label{proof: lemma MDDO}
\langle\mathrm{MDDO}(\bs{X}|Y)\bs\alpha,{\bs{\beta}}\rangle=- \mb{E}\Big(\langle \bs{X},\bs\alpha\rangle\langle \bs{X}',{\bs{\beta}}\rangle\|\Y-\Y'\|\Big) .
\end{align}
Then by the arbitrariness of $\bs\alpha,{\bs{\beta}}\in\mc H$, the proof is completed. 
\end{proof}

\section{Proof of Theorem \ref{theorem, MDDO and conditional mean independence}}

According to \eqref{proof: lemma MDDO}, one can get the following useful lemma.
\begin{lemma}\label{lemma, MDDO and FMDD}
Under Assumption $\ref{as:joint distribution assumption}$, for all ${\bs{\beta}}\in\mathcal H$, $\|{\bs{\beta}}\|=1$, we have
\begin{align*}
\langle \mathrm{MDDO}(\boldsymbol{X}|\Y)({\bs{\beta}}),{\bs{\beta}}\rangle &=- \mathbb E\Big[ \langle \boldsymbol{X},{\bs{\beta}}\rangle \langle \boldsymbol{X}',{\bs{\beta}}\rangle \|\Y-\Y'\|\Big]\\
&=- \mathbb E\Big[\big\langle\langle \boldsymbol{X},{\bs{\beta}}\rangle{\bs{\beta}},\langle \boldsymbol{X}',{\bs{\beta}}\rangle{\bs{\beta}}\big\rangle\|\Y-\Y'\|\Big].
\end{align*}
\end{lemma}
This conclusion links MDDO with functional martingale
difference divergence  (FMDD, \citealt{lee2020testing}). 
Next we give the following two lemmas to finish the proof of Theorem $\ref{theorem, MDDO and conditional mean independence}$.
\begin{lemma}\label{lem: Txx=0tuiTx=0}If $T$ is a positive semi-definite operator on a Hilbert space $\wt{\mathcal{H}}$, then for all $x\in\wt{\mathcal{H}}$, one has $\langle Tx,x\rangle=0\Longleftrightarrow Tx=0$.
\end{lemma}
\begin{proof}
`$\Longleftarrow$': It is obvious.

`$\Longrightarrow$': It is easy to check that $f(a,b)=\langle Ta,b\rangle$ $(a,b\in\wt{\mc H})$ is a 
positive semi-definite Hermitian form. Thus, for any $y\in\wt{\mathcal{H}}$, one can use Cauchy inequality to get
\[|\langle Tx,y\rangle|^2\leqslant\langle Tx,x\rangle\langle Ty,y\rangle=0\Longrightarrow \langle Tx,y\rangle=0.\]
By the arbitrariness of $y\in\wt{\mc H}$, one has $Tx=0$.
\end{proof}

Our proof of Theorem $\ref{theorem, MDDO and conditional mean independence}$ is mainly inspired by the following property of
FMDD in \cite{lee2020testing}.
\begin{lemma}[Proposition 1 of \cite{lee2020testing}]\label{lem:prop1inlee}
If $\E[\|\X\|+\|\Y\|]<\infty$ and $\E[\|\bs X\|\|\Y\|]<\infty$, then we have
\[\E[\langle \X,\X'\rangle\|\Y-\Y'\|]=0\Longleftrightarrow \E[\X|\Y]=0\quad\text{almost surely},\]
where $(\X',\Y')$ is an i.i.d. copy of $(\X,\Y)$.
\end{lemma}
\paragraph{Proof of Theorem $\ref{theorem, MDDO and conditional mean independence}$}
\begin{proof}
Clearly, (ii) is a direct consequence of Lemma $\ref{lemma, equivalence of two def of MDDO}$ and the following lemma.

\begin{lemma}[Lemma 15 in \citealt{chen2023optimality}]\label{lem:cov TX}
If $T$ is an operator defined on $\mc H_1\to\mc H_2$ where $\mc H_i,i=1,2$ is a Hilbert space. $\bs X\in\mc H_1$ is a random element satisfying $\mb E[\bs X]=0$ . Then we have $\mr{var}(T\bs X)=T\mr{var}(\bs X)T^*$.
\end{lemma}

Now we start  to prove (i).
 First, one has
\begin{align*}\mathrm{MDDO}(\boldsymbol{X}|\Y)=0 &\Longleftrightarrow \mathrm{MDDO}(\boldsymbol{X}|\Y)({\bs{\beta}})=0,\quad\forall{\bs{\beta}}\in\mb{S}_{\mathcal H};\\
\mathbb E[\boldsymbol{X}|\Y]=0~~\text{a.s.}&\Longleftrightarrow\langle\mb E[\boldsymbol{X}|\Y],{\bs{\beta}}\rangle{\bs{\beta}}=0~~\text{a.s.} \quad\forall{\bs{\beta}}\in\mb{S}_{\mathcal H},
\end{align*}
where $\mb{S}_\mc{H}=\{{\bs{\beta}}\in\mc H:\|{\bs{\beta}}\|=1\}$. Second, from Lemma $\ref{lem: Txx=0tuiTx=0}$, one knows that
\begin{align*}
\mathrm{MDDO}(\boldsymbol{X}|\Y)({\bs{\beta}})=0&\Longleftrightarrow\langle\mathrm{MDDO}(\boldsymbol{X}|\Y)({\bs{\beta}}),{\bs{\beta}}\rangle=0.
\end{align*}
 Then under Assumption $\ref{as:joint distribution assumption}$, by Lemma $\ref{lemma, MDDO and FMDD}$ and $\ref{lem:prop1inlee}$, one has
\begin{align*}
&\langle\mathrm{MDDO}(\boldsymbol{X}|\Y)({\bs{\beta}}),{\bs{\beta}}\rangle=0\Longleftrightarrow\mathbb E[\big\langle\langle \boldsymbol{X},{\bs{\beta}}\rangle{\bs{\beta}},\langle \boldsymbol{X}',{\bs{\beta}}\rangle{\bs{\beta}}\rangle\|\Y-\Y'\|]=0\\
&\qquad\qquad\qquad\qquad\qquad\Longleftrightarrow\mathbb E[\langle \bs X,{\bs{\beta}}\rangle{\bs{\beta}}|\Y]=\langle \mb E[\boldsymbol{X}|\Y],{\bs{\beta}}\rangle{\bs{\beta}}=0~~\text{a.s.}
\end{align*}
This finishes the proof of Theorem $\ref{theorem, MDDO and conditional mean independence}$.
\end{proof}

% {\color{blue}\paragraph{Proof of Lemma \ref{lem:cov TX} (Repeated)}
% \begin{proof}
% For any $\u_1,\u_2\in\mc H_2$, we have
% \begin{align*}
% &\left\langle  T\mr{var}(\vX)T^*\u_1,\u_2  \right\rangle=\left\langle  T\mb E[\vX\otimes\vX]T^*\u_1,\u_2  \right\rangle
% =\left\langle  \mb E[\vX\otimes\vX]T^*\u_1,T^*\u_2  \right\rangle    
% \end{align*}
% since $\mb E[\vX]=0$. By the definition of convariance operator and expectation, we have 
% \begin{align*}
% \left\langle  \mb E[\vX\otimes\vX]T^*\u_1,T^*\u_2  \right\rangle=&\left\langle  \mb E[\left\langle\vX,  T^*\u_1 \right\rangle       \vX            ],T^*\u_2  \right\rangle
% =\mb E[  \left\langle\vX,  T^*\u_1 \right\rangle      \left\langle \vX            ,T^*\u_2  \right\rangle].
% \end{align*}
% Similarly, we have
% \begin{align*}
%  \left\langle  \mr{var}(T\vX)\u_1,\u_2  \right\rangle=\left\langle  \mb E[T\vX\otimes T\vX]\u_1,\u_2  \right\rangle=\mb E[  \left\langle T\vX,  \u_1 \right\rangle      \left\langle T\vX            ,\u_2  \right\rangle].\\    
% \end{align*}
% Then the proof is completed by noticing the following
% \begin{align*}
% \mb E[  \left\langle T\vX,  \u_1 \right\rangle      \left\langle T\vX            ,\u_2  \right\rangle]=\mb E[  \left\langle\vX,  T^*\u_1 \right\rangle      \left\langle \vX            ,T^*\u_2  \right\rangle].
% \end{align*}
% \end{proof}}

\section{Proof of Lemma \ref{lemma: SE=GammaS}}
Recall the following fact in FSIR.
\begin{lemma}\label{lemma, direct result of linearity condition}~\\
Under Assumption $\ref{as:Linearity condition and Coverage condition}~ \boldsymbol{\mathrm{{i)}}}$, we have $\mathcal S_{\mathbb E(\boldsymbol{X}|\Y)}\subseteq \Gamma \mc S_{\Y|\bs X}\subseteq \mc H$.
\end{lemma}
It is a trivial generalization of    \cite[Theorem 2.1]{ferre2003functional} from univariate response to multivariate response.
\paragraph{Proof of Lemma $\ref{lemma: SE=GammaS}$}
\begin{proof}
First, we prove that $\mathcal{S}_{\mathbb{E}(\bs X|\Y)}^\perp\subseteq \mathrm{Im}\{\mathrm{var(\mb{E}(\bs X|\Y))}\}^\perp$. For any ${\bs{\beta}}\in\mathcal{S}_{\mathbb{E}(\bs X|\Y)}^\perp$, one has $\langle{\bs{\beta}},\mb{E}(\bs X|\Y)\rangle=0$ a.s. Then for any $\bs\alpha\in\mathcal{H}$, one can get
\begin{align*}
\langle{\bs{\beta}},\mathrm{var}(\mb{E}(\bs X|\Y))\bs\alpha\rangle&=\langle{\bs{\beta}},\mb E\lmi\mb{E}(\bs X|\Y)\otimes \mb{E}(\bs X|\Y)\rmi\bs\alpha\rangle\\
&=\mb E\big(\langle\mb{E}(\bs X|\Y),\bs\alpha\rangle\langle{\bs{\beta}},\mb{E}(\bs X|\Y)\rangle\big)=0,
\end{align*}
which means that ${\bs{\beta}}\in\mathrm{Im}\{\mathrm{var}(\mb{E}(\bs X|\Y))\}^\perp$. Moreover, one has
\begin{align*}\mathcal{S}_{\mathbb{E}(\bs X|\Y)}^\perp\subseteq \mathrm{Im}\{\mathrm{var}(\mb{E}(\bs X|\Y))\}^\perp
%&\Rightarrow\left(\mathcal{S}_{\mathbb{E}(\bs X|Y)}^\perp\right)^\perp\supseteq \left(\mathrm{Im}\{\mathrm{var(\mb{E}(\bs X|Y))}\}^\perp\right)^\perp\\&
\Longrightarrow\overline{\mathcal{S}_{\mathbb{E}(\bs X|\Y)}}\supseteq\overline{\mathrm{Im}}\{\mathrm{var}(\mb{E}(\bs X|\Y))\}.
\end{align*}
Thus, $\overline{\mathrm{Im}}\{\mathrm{var}(\mb{E}(\bs X|\Y))\}\subseteq\overline{\mathcal{S}_{\mathbb{E}(\bs X|\Y)}}\subseteq\overline{\Gamma\mathcal{S}_{\Y|\bs X}}$ by Lemma $\ref{lemma, direct result of linearity condition}$. According to Assumption $\ref{as:Linearity condition and Coverage condition}$ \textbf{ii)}, one can get
\[\mathrm{dim}\left(\overline{\mathrm{Im}}\{\mathrm{var}(\mb{E}(\bs X|\Y))\}\right)=\mathrm{dim}\left(\overline{\mathcal{S}_{\mathbb{E}(\bs X|\Y)}}\right)=\mathrm{dim}(\overline{\Gamma\mathcal{S}_{\Y|\bs X}})=d.\]
One can complete the proof since finite dimension subspaces are closed.
\end{proof}

\section{Proof of Theorem \ref{theorem, MDDO and IRS}}
\begin{proof}
For convenience, we abbreviate $\mathrm{MDDO}(\boldsymbol{X}|\Y)$ to ${M}$. According to Theorem $\ref{theorem, MDDO and conditional mean independence}$ and Lemma $\ref{lem: Txx=0tuiTx=0}$, one can get
\begin{align*}{\bs{\beta}}\in\mathcal S_{\mb E(\boldsymbol{X}|\Y)}^\perp&\Longleftrightarrow\langle {\bs{\beta}},\mathbb E(\boldsymbol{X}|\Y)\rangle=0~~\text{a.s.}\Longleftrightarrow\mathbb E(\langle {\bs{\beta}},\boldsymbol{X}\rangle|\Y)=0~~\text{a.s.}\\
&\Longleftrightarrow\mathrm{MDDO}(\langle {\bs{\beta}},\boldsymbol{X}\rangle|\Y)=0\Longleftrightarrow\langle {M}{\bs{\beta}},{\bs{\beta}}\rangle=0\\
&\Longleftrightarrow{M}{\bs{\beta}}=0\Longleftrightarrow {\bs{\beta}}\in\mathrm{null}(M)=\overline{\mathrm{Im}}(M)^\perp,
\end{align*}
which means that $\mathcal S_{\mb E(\boldsymbol{X}|\Y)}^\perp=\overline{\mathrm{Im}}(M)^\perp$ and $\overline{\mathcal S_{\mb E(\boldsymbol{X}|\Y)}}=\overline{\mathrm{Im}}(M)$.
One can complete the proof since finite dimension subspaces are closed.
\end{proof}
\section{Proof of Lemma \ref{lemma, way of estimate truncate central subspace}}
Before proving Lemma $\ref{lemma, way of estimate truncate central subspace}$, we give the following lemma.
\begin{lemma}\label{lem: colPBP equal colPB operator}
Assume that $P$ is a bounded linear operator from a Hilbert space $\wt{\mc H}$ to itself and $B$ is a positive semi-definite operator from $\wt{\mc H}$ to itself. 
Then we have $\overline{\mathrm{Im}}(PBP^*)=\overline{\mathrm{Im}}(PB)$.
\end{lemma}
\begin{proof}
It suffices to show that $\mnull(BP^*)=\mnull(PBP^*)$. First, since $B$ is positive semi-definite, one has $\langle x,PBP^*x\rangle = \langle P^*x,BP^*x \rangle\geqslant 0~(\forall x\in\wt{\mc H})$. Thus $PBP^*$ is a positive semi-definite operator on $\wt{\H}$.
For any $y\in\wt{\H}$, we have 
\begin{align*}
PBP^*y=0\overset{(a)}{\Longleftrightarrow}\langle y,PBP^*y\rangle = \langle P^*y,BP^*y \rangle=0\overset{(b)}{\Longleftrightarrow} BP^*y=0. 
\end{align*}
where $(a)$ and $(b)$ come from Lemma $\ref{lem: Txx=0tuiTx=0}$.
Thus $\mnull(PBP^*)=\mnull(BP^*)$.
\end{proof}

\paragraph{Proof of Lemma $\ref{lemma, way of estimate truncate central subspace}$}
\begin{proof}
For convenience, we abbreviate $\mathrm{MDDO}(\boldsymbol{X}|\Y)$ and $\mathrm{MDDO}(\boldsymbol{X}^{(m)}|\Y)$ to ${M}$ and $M_m$ respectively. 

By Corollary $\ref{corollary, MDDO and central subspace}$, one can get $\Gamma\mathcal{S}_{\Y|\boldsymbol{X}}=\mathrm{Im}(M)$. Thus,
\begin{align}\label{eq: corollary, MDDO and central subspace}
\Pi_m\Gamma\mathcal{S}_{\Y|\boldsymbol{X}}=\Pi_m\mathrm{Im}(M)=\mathrm{Im}(\Pi_m M).
\end{align}
It is easy to check that
\begin{align}
\Gamma_m&:=\mathrm{var}(\bs X^{(m)})=\Pi_m\Gamma\Pi_m=\Pi_m\Gamma=\Gamma\Pi_m=\sum\limits_{i=1}^m\lambda_i\phi_i\otimes\phi_i.\label{eq: Gamma m def}
\end{align}
On the one hand, by the definition of $\mathcal{S}^{(m)}_{{\Y|\boldsymbol{X}}}$ and $\Gamma_m$ (see \eqref{def: truncated central subspace} and \eqref{eq: Gamma m def}), one can get
\begin{align}\label{eq:Pim Gamma S}
\Pi_m\Gamma\mathcal{S}_{\Y|\boldsymbol{X}}&=\Pi_m\Gamma\Pi_m\mathcal{S}_{\Y|\boldsymbol{X}}=(\Pi_m\Gamma)(\Pi_m\mathcal{S}_{\Y|\boldsymbol{X}})=\Gamma_m\mathcal{S}^{(m)}_{{\Y|\boldsymbol{X}}}.
\end{align}
On the other hand, one has $\overline{\mathrm{Im}}(\Pi_m M)=\overline{\mathrm{Im}}(\Pi_m M\Pi_m)$ by Lemma $\ref{lem: colPBP equal colPB operator}$. Since $\Pi_m M$ and $\Pi_m M\Pi_m$ are both of finite rank, one can further get
\begin{align*}
\mathrm{Im}(\Pi_m M)&=\overline{\mathrm{Im}}(\Pi_m M)=\overline{\mathrm{Im}}(\Pi_m M\Pi_m)=\mathrm{Im}(\Pi_mM\Pi_m).
\end{align*}
Then according to Theorem $\ref{theorem, MDDO and conditional mean independence}$(ii), one has
\begin{align}
\mathrm{Im}(\Pi_m M)=\mathrm{Im}(\Pi_mM\Pi_m)=\mathrm{Im}(M_m).\label{eq:Pim span M}
\end{align}
Combining \eqref{eq:Pim Gamma S}, \eqref{eq:Pim span M} with \eqref{eq: corollary, MDDO and central subspace}, one has $\Gamma_m\mathcal{S}^{(m)}_{{\Y|\boldsymbol{X}}}=\mathrm{Im}\{M_m\}$.
Finally, one can get $ \Gamma_m^\dagger\mathrm{Im}\{M_m\}=\Gamma_m^\dagger\Gamma_m\mathcal{S}^{(m)}_{{\Y|\boldsymbol{X}}}=\Pi_m\mathcal{S}^{(m)}_{{\Y|\boldsymbol{X}}}=\mathcal{S}^{(m)}_{{\Y|\boldsymbol{X}}}$.
\end{proof}

\section{Wely Inequality for a Self-adjoint and Compact Operator}\label{ap:Wely inequality for self-adjoint and compact operators}
First, we show the following three results in standard functional analysis textbook.
\begin{lemma}[Spectral theorem]\label{thm: Spectral theorem}Let $\wt{\mathcal{H}}$ be a Hilbert space and $A:\wt{\mc{H}}\to\wt{\mc{H}}$ be a compact, self-adjoint operator. There is an at most countable orthonormal basis $\{\wt e_j\}_{j\in J}$ ($J=\{1,\cdots,n\}$ or $\mathbb{Z}_{\geqslant1}$) of $\wt{\mathcal{H}}$ and eigenvalues $\{\wt\lambda_j\}_{j\in J}$ with $|\wt\lambda_1|\geqslant|\wt\lambda_2|\geqslant\cdots\geqslant0$ converging to zero, such that
\begin{align*}
x=\sum_{j\in J}\langle x,\wt e_j\rangle \wt e_j;\qquad Ax=\sum_{j\in J}\wt\lambda_j\langle x,\wt e_j\rangle \wt e_j,\qquad x \in\wt{\mathcal{H}}.
\end{align*}
\end{lemma}

\begin{lemma}[Rayleigh's principle]\label{lem:Rayleigh operator}Let $A$ be a compact, self-adjoint operator. If $\{\wt e_j\}_{j\in J}$ and $\{\wt\lambda_j\}_{j\in J}$ are eigenvectors and eigenvalues define in Lemma $\ref{thm: Spectral theorem}$ respectively. Then
\[|\wt\lambda_1|=\mathop{\sup\limits_{\|u\|=1}}|\langle Au,u\rangle|;\qquad|\wt\lambda_n|=\mathop{\sup\limits_{\|u\|=1}}_{u\in\{\wt e_1,\cdots,\wt e_{n-1}\}^\perp}|\langle Au,u\rangle|~(n\geqslant 2).\]
\end{lemma}
\begin{lemma}[Minimax theorem]\label{lem:minimax operator}
Assume that $A$ is a positive semi-definite and compact operator with its eigenvalues $\{\wt\lambda_i\}$ ordered as $\wt\lambda_1\geqslant\dots\geqslant \wt\lambda_n\geqslant\dots\geqslant 0$, then
$$
\wt\lambda_n=\inf_{E_{n-1}}\sup_{x\in E_{n-1}^\perp,\|x\|=1}\langle Ax,x\rangle
$$
where $E_{n-1}$ with dimension $n-1$ is a closed linear subspace of $\wt{\mc H}$.
\end{lemma}
Then we give the Wely inequality for a self-adjoint and compact operator.
\begin{proposition}\label{prop: wely operator}
Let $M=N+R$ where $M$, $N$ and $R$ are three self-adjoint and compact operators defined on a Hilbert space $\wt{\mc H}$. Also, $M$ and $N$ are positive semi-definite with their respective eigenvalues $\{\mu_i\},\{\nu_i\}$ ordered as follows
\begin{align*}
M:\mu_1\geqslant\dots\geqslant \mu_n\geqslant\dots\geqslant 0;\qquad
N:\nu_1\geqslant\dots\geqslant \nu_n\geqslant\dots\geqslant 0,
\end{align*}
while $R$'s eigenvalues are $\{\rho_i\}$ ordered as follows:
\[R:|\rho_1|\geqslant\dots\geqslant |\rho_n|\geqslant\dots\geqslant 0.\]
Then the following inequalities hold: $|\mu_k-\nu_k|\leqslant|\rho_1|=\|R\| $, $k\geqslant1$.
\end{proposition}
\begin{proof}
From Lemma $\ref{lem:minimax operator}$, we have:
\[\mu_n=\inf_{E_{n-1}}\sup_{x\in E_{n-1}^\perp,\|x\|=1}\langle Mx,x\rangle;\qquad\nu_n=\inf_{E_{n-1}}\sup_{x\in E_{n-1}^\perp,\|x\|=1}\langle Nx,x\rangle,\]
where $E_{n-1}$ with dimension $n-1$ is a closed linear subspace of $\wt{\mc H}$.
By Lemma $\ref{lem:Rayleigh operator}$, we have:
$$
\sup_{\|u\|=1}|\langle Ru,u\rangle|=|\rho_1|.
$$
Since $\langle Mu,u\rangle=\langle Nu,u\rangle+\langle Ru,u\rangle$, for any $\|u\|=1$, we have:
$$
\langle Nu,u\rangle-|\rho_1|\leqslant\langle Mu,u\rangle \leqslant \langle Nu,u\rangle+|\rho_1|.
$$
Then for any given $n-1$ dimensional closed linear subspace of $\wt{\mc H}$, we conclude
\begin{equation}\label{eq:max ineq}
\sup_{u\in E_{n-1}^\perp,\|u\|=1}\langle Nu,u\rangle-|\rho_1|\leqslant\sup_{u\in E_{n-1}^\perp,\|u\|=1}\langle Mu,u\rangle\leqslant \sup_{u\in E_{n-1}^\perp,\|u\|=1}\langle Nu,u\rangle+|\rho_1|.
\end{equation}
Take the infimum with respective to $E_{n-1}$ in \eqref{eq:max ineq}, we have
\[\nu_n-|\rho_1|\leqslant\mu_n\leqslant \nu_n+|\rho_1|\]
by Lemma $\ref{lem:minimax operator}$.
\end{proof}
The next result is a direct corollary of Proposition $\ref{prop: wely operator}$.
\begin{corollary}\label{coro:wely ineq operator}
Let $M$ and $N$ be two self-adjoint, positive semi-definite and compact operators defined on a Hilbert space $\wt{\mc H}$ with their respective eigenvalues $\{\mu_i\},\{\nu_i\}$ ordered as follows
\begin{align*}
M:\mu_1\geqslant\dots\geqslant \mu_n\geqslant\dots\geqslant 0\quad\text{and}\quad
N:\nu_1\geqslant\dots\geqslant \nu_n\geqslant\dots\geqslant 0.
\end{align*}
Then the following inequalities hold: $|\mu_k-\nu_k|\leqslant\|M-N\| $, $ k\geqslant1$.
\end{corollary}

\section{Proof of Proposition \ref{prop:bound hatMmd Mm}}
Before proving Proposition $\ref{prop:bound hatMmd Mm}$, we give the following conclusion, whose proof is deferred to the end of this section.
\begin{proposition}\label{proposition, concentration of MDDO}
Under Assumptions $\ref{as:joint distribution assumption}$ and $\ref{assumption: sub-Gaussian}$, for all $\gamma\in(0,1/2)$, there exist positive constants $D_0=D_0(\gamma,\sigma_0,\sigma_1)$, $D_1=D_1(\sigma_1)$, $D_2=D_2(\sigma_0,\sigma_1)$ and $n_0=n_0(\gamma,\sigma_0,\sigma_1)$ such that for all $n\geqslant n_0$ and
\[C\in \l D_0n^{\frac{2\gamma}{5}}-\ln\l D_1m^2n \r,D_2 n^{\frac{1}{5}}-\ln\l D_1m^2n \r \rmi,\]
we have
\begin{equation*}
\mathbb{P}\l\left\|\wh M_m- M_m\right\| <\l \frac{C+\ln( D_1m^2n)}{D_2}\r^{\frac52}\frac{12m}{\sqrt n}\r\geqslant 1-\exp(- C).
\end{equation*}
\end{proposition}
\paragraph{Proof of Proposition $\ref{prop:bound hatMmd Mm}$}
\begin{proof}

Using Corollary $\ref{coro:wely ineq operator}$, one can get
$
\lambda_i\l\wh M_m\r\leqslant \lno\wh M_m-M_m\rno +\lambda_i\l M_m\r
$. 
Since $\rank(M_m)=d$, one can get $\lambda_i(M_m)=0,~i\geqslant d+1$. Thus by Proposition $\ref{proposition, concentration of MDDO}$, one has
\begin{align}\label{eq:lambdai hat Mm upper bound}
\mathbb{P}\l\lambda_{d+1}(\wh M_m)<\l \frac{C+\ln\l D_1m^2n\r}{D_2}\r^{\frac52}\frac{12m}{\sqrt n}\r\geqslant 1-\exp(- C)\qquad(i\geq d+1). 
\end{align}
Notice that 
\begin{align*}\lno\wh M_m^d- M_m\rno &\leqslant\lno M_m-\wh M_m\rno +\lno\wh M_m-\wh M_m^d\rno ;\\
\lno\wh M_m-\wh M_m^d\rno &=\left\|\sum_{i=d+1}^\infty\wh\mu_i\wh\gamma_i\otimes \wh\gamma_i\right\| =\widehat{\lambda}_{d+1}=\lambda_{d+1}(\widehat{M}_m)
\end{align*}
by \eqref{wh M_m spectral decomposition}.
Then combing Proposition $\ref{proposition, concentration of MDDO}$ with \eqref{eq:lambdai hat Mm upper bound} can complete the proof.
\end{proof}

\paragraph{Proof of Proposition \ref{proposition, concentration of MDDO}}
\begin{proof}
Note that $\boldsymbol{X}^{(m)}=\sum\limits_{j=1}^m\langle \boldsymbol{X},\phi_j\rangle\phi_j$, then a simple calculation leads to
\begin{align*}
M_m&=-\sum_{i,j=1}^m\mathbb E\big[\langle \boldsymbol{X},\phi_i\rangle\langle \boldsymbol{X}',\phi_{j}\rangle\|\Y-\Y'\|\big]\phi_i\otimes\phi_j;\\
\widehat{M}_m&=-\sum_{i,j=1}^m\frac1{n^2}\sum_{k,\ell=1}^n\langle \boldsymbol{X}_k,\phi_i\rangle\langle \boldsymbol{X}_\ell,\phi_j\rangle\|\Y_k-\Y_\ell\|\phi_i\otimes\phi_j.
\end{align*}

For a operator $\Gamma'$ that can be expanded as $\Gamma':=\sum\limits_{i,j=1}^\infty a_{ij}\phi_i\otimes\phi_{j}$, let us define its maximal norm as $\|\Gamma'\|_{\mathrm{max}}=\sup\limits_{i,j}|a_{ij}|$.

\begin{lemma}\cite[Theorem 1]{mai2021slicing}\label{lemma, concentration of MDDOnm}
Under Assumptions $\ref{as:joint distribution assumption}$ and $\ref{assumption: sub-Gaussian}$, for all
$\gamma\in(0,1/2)$, there exist positive
constants $C_0=C_0(\gamma,\sigma_0,\sigma_1)$, $C_1=C_1(\sigma_1)$, $C_2 = C_2(\sigma_0;\sigma_1)$ and $n_0 = n_0(\gamma,\sigma_0,\sigma_1)$
such that for all $n\geqslant n_0$ and $\varepsilon\in(C_0 n^{-(1/2-\gamma)},1]$, we have
\begin{equation*}
\mathbb{P}\l\lno \widehat{M}_m-M_m\rno_{\max}>12\varepsilon\r\leqslant C_1 m^2n\exp\l- C_2\l\varepsilon^2 n\r^{1/5}\r.
\end{equation*}
\end{lemma}
\noindent Since $\lno\widehat{M}_m-M_m\rno \leqslant m\lno\widehat{M}_m-M_m\rno_{\mathrm{max}}$, one has
\begin{equation*}
\mathbb{P}\l\lno\widehat{M}_m-M_m\rno >12m\varepsilon\r\leqslant C_1 m^2n\exp\l-C_2\l\varepsilon^2 n\r^{1/5}\r.
\end{equation*}
Let $C=C_2\l\ve^2n\r^{1/5}-\ln\l C_1m^2n\r$ satisfying 
\begin{align*}
C\in\l C_2C_0^{2/5}n^{2\gamma/5}-\ln\l C_1m^2n\r,C_2n^{1/5}-\ln\l C_1m^2n\r\rmi,
\end{align*}
then one has
\begin{equation*}
\mathbb{P}\l\lno\widehat{M}_m-M_m\rno \leqslant\l \frac{C+\ln\l C_1m^2n\r}{C_2}\r^{\frac52}\frac{12m}{\sqrt{n}}\r>1- \exp(- C).
\end{equation*}
Then in order to complete the proof, one only need to choose $D_0$, $D_1$ and $D_2$ to be $C_2C_0^{2/5}$, $C_1$ and $C_2$ respectively. 
\end{proof}

\section{Properties of Sub-Gaussian Random Vectors}
We first review the definition of sub-Gaussian random variables.
\begin{definition}[Sub-Gaussian random variable and its upper-exponentially bounded constant]\label{def:sub gaussian variable}
A random variable $X$ is called a sub-Gaussian random variable if $X$ satisfies one of the following equivalent properties:
\begin{itemize}
 \item[1).] Tails. $\P(|X|>t)\leqslant \exp(1-t^{2}/K^{2}_{1})$ for any $t>0$;
 \item[2).] Moments. $\E[|X|^{p}]^{1/p}\leqslant K_{2}\sqrt{p}$ for any $p\geqslant 1$;
 \item[3).]Super-exponential moment: $\E[\exp(X^{2}/K^{2}_{3})]\leqslant \mr{e}$.

\noindent Moreover, if $\E[X]=0$, then the properties $1)-3)$ are also equivalent to the following one:
\item[4).] Moment generating function: $\E[\exp(tX)]\leqslant \exp(t^{2}K^{2}_{4})$ for all $t\in\R$.
\end{itemize}
Here $K_1$, $K_2$, $K_3$ and $K_4$ are four constants.
$K$ is called an upper-exponentially bounded constant of $X$ if 
$K\geqslant \max\{K_{1},K_{2},K_{3},K_{4}\}$.
\end{definition}
\begin{definition}[Sub-Gaussian random vector and its upper-exponentially bounded constant]\label{def,sub-Gaussian random vector,upper-exponentially bounded constant}
 ${X}\in\R^m$ is called a sub-Gaussian random vector if for all $x\in\R^m$, one-dimensional marginal $\langle{X},x\rangle$ is sub-Gaussian random variable. $K$ is called an upper-exponentially bounded constant of $X$ if $K$ satisfies:
 \begin{align*}
K\geqslant \sup_{x\in\mb{S}^{m-1}}K(\langle X,x\rangle) 
 \end{align*}
 where $K(\langle X,x\rangle)$ denotes an upper-exponentially bounded constant of $\langle X,x\rangle$.
Moreover, $K$ is called a uniform (about $m$) upper-exponentially bounded constant of $X$ if $K$ satisfies:
 \begin{align*}
K\geqslant \sup_m\sup_{x\in\mb{S}^{m-1}}K\l \langle X,x\rangle\r.
 \end{align*}
Furthermore, $X$ is called a uniform (about $m$) sun-Gaussian random vector.
 \end{definition}
The following is an application of sub-Gaussian random vectors.
\begin{lemma}[\citealt{vershynin2010introduction}]\label{lem:esgrm}
 Let $\M=[\bs m_1~\cdots~\bs m_n]$ be an $m\times n$ matrix ($n>m$) whose columns $\m_{i}$ are 
 independent centered sub-Gaussian random vectors with 
 covariance matrix $\mathbf{I}_{m}$. Let $\sigma^{+}_{\min}(\M)$ and $\sigma_{\max}(\M)$ be the infimum and supremum of positive singular values of $\M$ respectively. Then, for any $t>0$, with probability at least $1-2\exp(- C^{\prime}t^{2})$, we have
 \begin{equation*}
 \sqrt{n}-C_0\sqrt{m}-t\leqslant \sigma^{+}_{\min}(\M)\leqslant \sigma_{\max}(\M)\leqslant \sqrt{n}+C_0\sqrt{m}+t
 \end{equation*}
 where $C'$ and $C_0$ are two positive constants depending only on $K(\bs m_1)$:
 the upper-exponentially bounded constant of $\bs m_1$.
\end{lemma}
\noindent Let $t=\sqrt m$, then one can easily get
\begin{align}\label{equation, min max eval}
\begin{split}
\lambda_{\max}\left(\frac1n \M\M^\top\right)\leqslant \left(1+\frac{(C_0+1)\sqrt m}{\sqrt n}\right)^2;\\
\lambda_{\min}^+\left(\frac1n \M\M^\top\right)\geqslant \left(1-\frac{(C_0+1)\sqrt m}{\sqrt n}\right)^2, 
\end{split}
\end{align}
with probability at least $1-2\exp(- C'm)$ where $\lambda^{+}_{\min}(\cdot)$ and $\lambda_{\max}(\cdot)$ stands for the infimum and supremum of the positive spectrum respectively.

\begin{lemma}\label{lemma, estiamtion error of inverse sample cov}
Assume that $\x_1,\x_2,...,\x_n$ are $n$ i.i.d. samples from an $m$-dimensional centered sub-Gaussian vector with an invertible covariance matrix $\Sigma$. Let $\wh\Sigma:=\frac1n\sum_i \x_i\x_i^\top$.
Then there exists a positive constant $n_1'=n_1'(K(\bs m_1),c_1)$ ($c_1$ is defined in \eqref{eq: m n relationship}), such that when $n\geqslant n_1'$, we have
\begin{align*}
\lno\wh{\Sigma}-\Sigma\rno\hspace{-1.5mm}&\leqslant (C_0+2)^2\lambda_{\max}(\Sigma)\sqrt{\frac mn}~~\text{and}~~ \lno\wh{\Sigma}^{-1}-\Sigma^{-1}\rno\hspace{-1.5mm}\leqslant \frac{4(C_0+2)^2}{\lambda_{\min}(\Sigma)}\sqrt{\frac mn},
 \end{align*}
 with probability at least $1-2\exp(- C'm)$, where $C_0$ is defined in Lemma $\ref{lem:esgrm}$.
\end{lemma}
\begin{proof}
Let $\x_i=\Sigma^{\frac12}\m_i$ and $\bs{M}=[\bs m_1~\cdots~\bs m_n]$ where $\m_i$ is a centered sub-Gaussian random vector with covariance $\mathbf I_m$. Then one has 
\begin{align*}
\lno\wh\Sigma-\Sigma\rno&\leqslant\lno\Sigma^{\frac12}\rno\cdot\left\|\frac1n \M\M^\top-\mathbf I\right\|\cdot\lno\Sigma^{\frac12}\rno\\
&= \lambda_{\max}(\Sigma)\cdot\left[\lambda_{\max}\left(\frac1n \M\M^\top\right)-1\right]
\end{align*}
and 
\begin{align*}
\lno\wh{\Sigma}^{- 1}-\Sigma^{- 1}\rno
&\leqslant \lno\Sigma^{-\frac12}\rno\cdot\left\|\frac1n \M\M^\top-\mathbf I\right\|\cdot\lno\l\frac1n \M\M^\top\r^{-1}\rno\cdot\lno\Sigma^{-\frac12}\rno\\
&=\frac{1}{\lambda_{\min}(\Sigma)}\left[\lambda_{\max}\left(\frac1n \M\M^\top\right)-1\right]\cdot\lambda_{\min}\left(\frac1n \M\M^\top\right)^{-1}.
\end{align*}
By \eqref{equation, min max eval}, it is easy to check that
\begin{align*}&\lambda_{\max}\left(\frac1n \M\M^\top\right)-1\leqslant\left(1+\frac{(C_0+1)\sqrt m}{\sqrt n}\right)^2-1\leqslant\frac{(C_0+2)^2\sqrt m}{\sqrt n};\\
&\lambda_{\min}\left(\frac1n \M\M^\top\right)\geqslant \left(1-\frac{(C_0+1)\sqrt m}{\sqrt n}\right)^2\geqslant \frac14~\text{for}~n\geqslant [2(C_0+1)]^{\frac2{1-c_1}},
\end{align*}
with probability at least $1-2\exp(- C'm)$. Thus the proof is completed by choosing $n_1'(C_0,c_1):=[2(C_0+1)]^{\frac{2}{1-c_1}}$. 
\end{proof}

\section{Proof of Proposition \ref{prop:concentration Gammam dag Mmd}}\label{ap:concentration inequality}
We first give the following lemma whose proof is deferred to the end of this section.
\begin{lemma}\label{lem:PimTPimtoT}If $T$ is of finite rank, then we have $\lim\limits_{m\to \infty}\|\Pi_m T\Pi_m-T\| =0$.
\end{lemma}
A direct corollary of this lemma is as follows.
\begin{corollary}\label{lemma, M go to Mm}
%For any $\varepsilon>0$, one has $\|M-M_m\| <\varepsilon$ when $m$ is sufficiently large.
Under Assumptions $\ref{as:joint distribution assumption}$ and $\ref{as:Linearity condition and Coverage condition}$, we have $\lim\limits_{m\to\infty}\|M-M_m\| =0$.
\end{corollary}
\noindent We denote by $m_M(\varepsilon)$ the minimal integer $m_M$ satisfying $\|M-M_m\| \leqslant \varepsilon$ for all $m\geqslant m_M$.

Proposition $\ref{prop:concentration Gammam dag Mmd}$ is a direct corollary of the following Proposition.
\begin{proposition}
\label{prop:bound of finite estimate}
 Suppose that Assumptions $\ref{as:joint distribution assumption}$ to $\ref{assumption: rate-type condition}$ hold, then $\forall \gamma\in(0,1/2)$, there exist positive constants
 \begin{align*}
 n_1=n_1(\gamma,\sigma_0,\sigma_1,\bs K,m_M(1),c_1),\quad D_3=D_3(\|M\| ,\wt C,\bs K) 
 \end{align*}
and $C'=C'(\bs K)$
, such that when $n\geqslant n_1$, we have
\begin{equation*}
\begin{aligned}
\mb P\l \lno\widehat\Gamma_m^\dagger \widehat M_m^d-\Gamma_m^\dagger M_m\rno  \leqslant \left[\frac{C+\ln(D_1m^2n)}{D_2}\right]^{\frac52}\frac{24m^{\alpha_1+1}}{\wt C\sqrt n}+D_3\frac{m^{(2\alpha_1+1)/2}}{n^{1/2}} \r&\\
\geqslant 1-\exp(- C)-2\exp(- C'm).&
\end{aligned}
\end{equation*}
Here $D_1,D_2$ and $C$ are defined in Proposition $\ref{prop:bound hatMmd Mm}$ and $\bs K$ is the uniform upper-exponentially bounded constant of $(\sqrt{\lambda_1}w_1,\dots,\sqrt{\lambda_m}w_m)$. 
\end{proposition}
\begin{proof}
By triangle inequality, one has
\begin{align*}
&\lno\widehat{\Gamma}_m^\dagger \widehat M_m^d-\Gamma_m^\dagger M_m\rno 
=\lno\widehat\Gamma_m^\dagger \widehat M_m^d-\wh\Gamma_m^\dagger M_m+\wh\Gamma_m^\dagger M_m-\Gamma_m^\dagger M_m\rno 
\\&\qquad\leqslant\lno\Gamma_m^\dagger\rno \cdot \lno\widehat M_m^d-M_m\rno +\lno\widehat\Gamma_m^\dagger-\Gamma_m^\dagger\rno \cdot \lno M_m\rno .
\end{align*}
Thus one can bound $\lno\Gamma_m^{\dag}M_m-\widehat\Gamma_m^{\dag}\widehat M_m^d\rno $ by bound $\lno\Gamma_m^\dagger\rno $, $\lno\widehat\Gamma_m^\dagger-\Gamma_m^\dagger\rno $, $\lno\widehat M_m^d-M_m\rno $ and $\lno M_m\rno $ respectively.
\begin{itemize}
 \item\textbf{Bound of $\lno\Gamma_m^\dagger\rno $}: By Assumption $\ref{assumption: rate-type condition}$, one has 
\begin{align}\label{eq:bound Gammam dagger}
\lambda_j\geqslant \wt C j^{-\alpha_1}\Rightarrow\lno\Gamma_m^\dagger\rno =\lambda_m^{-1}\leqslant \wt{C}^{-1} m^{\alpha_1}. 
\end{align} 
 \item\textbf{Bound of $\lno\widehat\Gamma_m^\dagger-\Gamma_m^\dagger\rno $}:
 Let us define $\mc H_m:=\mathrm{span}\{\phi_1,\dots,\phi_m\}$ where $\{\phi_i\}$ is introduced in Equation $\eqref{eq:X expansion}$. It is easy to check that
$\lno\widehat\Gamma_m^\dagger-\Gamma_m^\dagger\rno =\lno(\widehat\Gamma_m^\dagger-\Gamma_m^\dagger)|_{\mc H_m}\rno $ since $\l\widehat\Gamma_m^\dagger-\Gamma_m^\dagger\r{\bs{\beta}}=0$ for any ${\bs{\beta}}\in\mc{H}_m^\perp$. 
Because $\l\widehat\Gamma_m^\dagger-\Gamma_m^\dagger\r|_{\mc H_m}$ can be represented by matrix $\widehat{\Sigma}^{-1}-\Sigma^{-1}$ defined in Lemma $\ref{lemma, estiamtion error of inverse sample cov}$ under orthonormal basis $\{\phi_i\}_{i=1}^m$, one can get $\lno\widehat\Gamma_m^\dagger-\Gamma_m^\dagger\rno =\|\widehat{\Sigma}^{-1}-\Sigma^{-1}\|$.
Similarly, one can also get $\lno\Gamma_m^\dagger\rno =\lno\Sigma^{-1}\rno=\lambda_{\min}^{-1}(\Sigma)$. Thus, by Lemma $\ref{lemma, estiamtion error of inverse sample cov}$ one has
\[\mb P\l\lno\widehat\Gamma_m^\dagger-\Gamma_m^\dagger\rno \leqslant {4(C_0+2)^2}\lno\Gamma^{\dag}_m\rno \sqrt{\frac mn}\r\geqslant 1-2\exp(- C'm)\]
for sufficiently large $n\geqslant n_1'(\bs K,c_1)$
. Combing with $\lno\Gamma_m^\dagger\rno \hspace{-1mm}\leqslant \wt{C}^{-1} m^{\alpha_1}$, one can get
\begin{equation}\label{eq: distance hat gamma m dagger hat gamma m dagger}
\mb P\l\lno\widehat\Gamma_m^\dagger-\Gamma_m^\dagger\rno \leqslant \frac{4(C_0+2)^2m^{(2\alpha_1+1)/2}}{\wt Cn^{1/2}}\r\geqslant 1-2\exp(- C'm)
\end{equation}
for sufficiently large $n\geqslant n_1'(\bs K,c_1)$.
 \item\textbf{Bound of $\lno\widehat M_m^d-M_m\rno $}:
 See Proposition $\ref{prop:bound hatMmd Mm}$.
 \item \textbf{Bound of $\lno M_m\rno $}: By Corollary $\ref{lemma, M go to Mm}$, $\|M-M_m\| \leqslant 1$ for sufficiently large $m\geqslant m_M(1)$. Then by triangle inequality, one can get
\[\|M_m\| -\|M\| \leqslant \|M-M_m\| \leqslant 1.\]
Hence,
\begin{align}\label{eq:Mm leq M C}
\|M_m\| \leqslant \|M\| +1.
\end{align}
\end{itemize}
Combing \eqref{eq:bound Gammam dagger}, \eqref{eq: distance hat gamma m dagger hat gamma m dagger}, Proposition $\ref{prop:bound hatMmd Mm}$ with \eqref{eq:Mm leq M C}, one can choose $D_3$ and $n_1$ to be $\frac{4(C_0+2)^2(\|M\| +1)}{\wt C}$ and $\max\{n_0,n_1'(\bs K,c_1),m_M(1)^{1/c_1}\}$ respectively to
complete the proof where $n_0$ is defined in Proposition $\ref{prop:bound hatMmd Mm}$.
\end{proof}

\paragraph{Proof of Lemma \ref{lem:PimTPimtoT}}
\begin{proof}By the triangle inequality and compatibility of operator norm, one has
\begin{align*}
\|\Pi_m T\Pi_m-T\| &\leqslant\|\Pi_mT\Pi_m-\Pi_mT\| +\|\Pi_mT-T\| \\
&\leqslant\|(\Pi_m-I)T^*\| +\|(\Pi_m-I)T\| 
\end{align*}
where $I=\sum\limits_{i=1}^\infty\phi_i\otimes\phi_i$ for $\{\phi_i\}_{i\in\mb{Z}_{\geqslant 1}}$ defined in \eqref{eq:X expansion} being an orthonormal basis of $\mc H$. 
% Since the adjoint of $M(\Pi_m-I)$ is $(\Pi_m-I)M$, we have
% \begin{align*}&\|M(\Pi_m-I)\| +\|(\Pi_m-I)M\| \\
% =&
% \end{align*}

Since $T$ is of finite rank, let us assume that $\{e_i\}_{i=1}^k$ is an orthonormal basis of $\mathrm{Im}(T)$ where $k=\mr{rank}(T)$. For any ${\bs{\beta}}\in\mathcal{H}$ such that $\|{\bs{\beta}}\|=1$, one has $\|T{\bs{\beta}}\|\leqslant\|T\| \|{\bs{\beta}}\|=\|T\| $, so one can assume that $T{\bs{\beta}}\in\mathrm{Im}(T)$ admits the following expansion under basis $\{e_i\}_{i=1}^k$:
\[T{\bs{\beta}}=\sum_{i=1}^k b_ie_i,\quad \sum_{i=1}^k b^2_i\leqslant\|T\| ^2<\infty.\]
Thus
\[\|(I-\Pi_m)T{\bs{\beta}}\|=\left\|\sum_{i=1}^k(I-\Pi_m) b_ie_i\right\|\leqslant\sum_{i=1}^k |b_i|\cdot\|(I-\Pi_m) e_i\|.\]
Clearly, $\|(\Pi_m-I)\alpha\|~(\forall\alpha\in\H)$ tends to $0$ as $m\to\infty$ since 
\[(I-\Pi_m)\alpha=\left(\sum_{i={m+1}}^\infty\phi_i\otimes\phi_i\right)\left(\sum\limits_{i=1}^\infty c_i\phi_i\right)=\sum_{i=m+1}^\infty c_i\phi_i\xrightarrow{m\to\infty} 0\]
where we have assumed that $\alpha=\sum\limits_{i=1}^\infty c_i\phi_i$ .

Thus $\forall\varepsilon>0$, there exists some $N_i>0$ such that $\forall m> N_i$ one has $\|(\Pi_m-I)e_i\|<\varepsilon$, $(\forall i=1,...,k)$. Let $N=\max\{N_1,\cdots,N_k\}$, then $\forall m>N$ one has
\[\|(I-\Pi_m)T{\bs{\beta}}\|\leqslant\sum_{i=1}^k |b_i|\cdot\|(I-\Pi_m) e_i\|\leqslant\sum_{i=1}^k |b_i|\varepsilon\leqslant k\varepsilon\|T\| ,\]
which means that $\forall m>N$, one has
\begin{align*}
\|(\Pi_m-I)T\| &=\sup_{\|{\bs{\beta}}\|=1}\|(\Pi_m-I)T{\bs{\beta}}\|\leqslant k\varepsilon\|T\| . 
\end{align*}
Thus $\lim\limits_{m\to\infty}\|(\Pi_m-I)T\| =0$. 

Similarly, one can also get $\lim\limits_{m\to\infty}\|(\Pi_m-I)T^*\| =0$. Then the proof of Lemma $\ref{lem:PimTPimtoT}$ is completed.
\end{proof}

\section{Sin Theta Theorem}\label{ap:Sin Theta theorem}
\subsection{Sin Theta Theorem for Self-adjoint Operators}
\begin{lemma}[Proposition 2.3 in \cite{seelmann2014notes}]\label{lemma, sin theta of infinite dimension operator}
Let $B$ be a self-adjoint operator on a separable Hilbert space $\widetilde{\mathcal{H}}$, and let ${V}\in\mathcal{L}(\widetilde{\mathcal{H}})$ be another self-adjoint operator where $\mathcal{L}\left(\widetilde{\mc H}\right)$ stands for the space of bounded linear operators from a Hilbert space $\widetilde{\mc H}$ to $\widetilde{\mc H}$.
Write the spectra of $B$ and $B+V$ as \[\mathrm{spec}( B)=\sigma\cup\Sigma\quad\text{and}\quad \mathrm{spec}( B+ V)=\omega\cup\Omega
\]
with $\sigma\cap\Sigma=\varnothing=\omega\cap\Omega$, and suppose that there is $\widehat d>0$ such that
\[\mathrm{dist}(\sigma,\Omega)\geqslant \widehat d\quad\text{and}\quad\mathrm{dist}(\Sigma,\omega)\geqslant \wh d\]
where $\mathrm dist(\sigma,\Sigma):=\min\{|a-b|:a\in\sigma,b\in\Omega\}$.
Then it holds that
\[\|P_{{B}}(\sigma)-P_{{B}+{V}}(\omega)\| \leqslant\frac\pi2\frac{\| V\| }{\wh d}\]
where $P_{ B}(\sigma)$ denotes the spectral projection for $ B$ associated with $\sigma$, i.e., 
\[P_{B}(\sigma):=\frac{1}{2\pi\mathrm{i}}\oint_{\gamma}\frac{\mathrm{d}z}{z-B},\]
where $\gamma$ is a contour on $\mathbb{C}$ that encloses $\sigma$ but no other elements of $\mathrm{spec}( B)$.
\end{lemma}
\begin{remark}
We note that, 
if further $ B$ is compact, 
the spectral projection coincide with projection operator onto the closure of the space spanned by the eigenfunctions associated with the eigenvalues in $\sigma$. 
% For more details, see, e.g., Remark 1 in \cite{chen2023optimality}.

Specifically, if $B$ is compact, by the spectral decomposition theorem one has
\[B=\sum_{i=1}^\infty\mu_ie_i\otimes e_i\quad\text{and}\quad(z- B)^{-1}=\sum_{i=1}^\infty(z-\mu_i)^{-1}e_i\otimes e_i,\]
where $\mr{spec}(B):=\{\mu_i\}_{i=1}^\infty$ satisfies $|\mu_i|\xrightarrow{i\to\infty} 0$.
Then $\forall v\in \mathcal{H}$, it holds that
\begin{align*}P_{B}(\sigma)v&=\frac{1}{2\pi\mathrm{i}}\oint_{\gamma}({z-B})^{-1}v~{\mathrm{d}z}=\frac{1}{2\pi\mathrm{i}}\oint_{\gamma}\sum_{i=1}^\infty(z-\mu_i)^{- 1}\langle e_i,v\rangle e_i~{\mathrm{d}z}\\
&=\sum_{i=1}^\infty\left[\left(\frac{1}{2\pi\mathrm{i}}\oint_{\gamma}(z-\mu_i)^{-1}~{\mathrm{d}z}\right)\langle e_i,v\rangle e_i\right]=\sum_{i\in\{i:\mu_i\in\sigma\}}\langle e_i,v\rangle e_i.
\end{align*}
In particular, if $\sigma=\mr{spec}(B)\backslash\{0\}$, then $P_{B}(\sigma)$ is the projection operator onto the $\overline{\mathrm{Im}}(B)$.
\end{remark}

Splitting eigenvalues into nonzero part and zero part yields the following useful corollary.
\begin{corollary}\label{cor: sin theta self adjoint}
Let $B$ and $B'$ be two positive semi-definite {and compact} operators with finite rank on a separable Hilbert space $\widetilde{\mathcal{H}}$. Let $\lambda_{\min}^+( B)$ and $\lambda_{\min}^+(B')$ be the infimum of the positive eigenvalues of ${B}$ and ${B}'$ respectively. Then we have
\[\left\|P_{ B}-P_{ B'}\right\| \leqslant\frac\pi2\frac{\| B- B'\| }{\min\{\lambda_{\min}^+( B),\lambda_{\min}^+( B')\}}.\]
\end{corollary}
\subsection{Sin Theta Theorem for General Operators}
When ${B}$ and ${V}$ in Lemma $\ref{lemma, sin theta of infinite dimension operator}$ are not self-adjoint, we use the symmetrization trick, which mainly depends on the following Lemma.
\begin{lemma}\label{lem:projection equality}
$P_A=P_{AA^*}$ for any bounded linear operator $A$ from a Hilbert space $\wt\H$ to $\wt\H$. Especially, $P_A=P_{AA^{\top}}$ for any matrix $A$.
\end{lemma}
\begin{proof}First we show that the null space of  $A^*$ is the same as the null space of $AA^*$.
On the one hand, 
\[x\in\mathrm{null}(A^*)\Longrightarrow
A^*x=0\Longrightarrow AA^*x=0\Longrightarrow x\in\mathrm{null}(AA^*); 
\]
One the other hand,
\begin{align*}x\in\mathrm{null}(AA^*)&\Longrightarrow
AA^*x=0\Longrightarrow \langle x,AA^*x\rangle=\langle A^*x,A^*x\rangle=\|A^*x\|^2=0\\
&\Longrightarrow A^*x=0\Longrightarrow x\in\mathrm{null}(A^*).
\end{align*}
Hence, we have $\mathrm{null}(A^*)=\mathrm{null}(AA^*)$. Take the orthogonal complement of the both sides of this equality, we can get
\[\mathrm{null}(A^*)^{\perp}=\mathrm{null}(AA^*)^{\perp}\Longrightarrow {\mathrm{Im}(A)}={\mathrm{Im}(AA^*)}.\]
\end{proof}
Then we have the following Sin Theta theorem for general operator.
\begin{lemma}\label{lemma, sin theta of nonadjoint operator}
Let $ B,B'\in\mathcal{L}(\widetilde{\mathcal{H}})$ be two compact operators (not necessarily self-adjoint) with finite rank.
Then we have
\begin{align*}
\left\|P_{ B}-P_{ B'}\right\| &\leqslant\frac\pi2\frac{\| B B^*- B'B'^*\| }{\min\left\{\sigma_{\min}^+( B)^2,\sigma_{\min}^+(B')^2\right\}}\\
&\leqslant \frac\pi2\frac{\| B- B'\| ^2+2\| B- B'\| \| B'\| }{\min\left\{\sigma_{\min}^+( B)^2,\sigma_{\min}^+( B')^2\right\}}.
\end{align*}
\end{lemma}
\begin{proof}By Lemma $\ref{lem:projection equality}$, one can get $\left\|P_{ B}-P_{ B'}\right\| =\left\|P_{ B B^*}-P_{ B' B'^*}\right\| $.
Since $ BB^*, B'B'^*$ are both self-adjoint and compact, by Lemma $\ref{cor: sin theta self adjoint}$, one has
\begin{align*}
\left\|P_{ B B^*}-P_{ B' B'^*}\right\| \leqslant \frac{\pi}{2}\frac{\| B B^*- B' B'^*\| }{\min\left\{\lambda_{\min}^+\left( B B^*\right),\lambda_{\min}^+\left( B' B'^*\right)\right\}}.
\end{align*}
Then the proof is completed in view of the following inequality:
% of $\| B B^*- B' B'^*\| $:
\begin{align}
\left\| B B^*- B' B'^*\right\| &= \|( B- B')( B- B')^*\hspace{-0.5mm}+\hspace{-0.5mm}( B-B')(B')^*\hspace{-0.5mm}+\hspace{-0.5mm} B'( B- B')^*\| \nonumber\\
&\leqslant \| B- B'\| ^2+2\| B- B'\| \| B'\| . \label{eq:sy ineq}
\end{align}
\end{proof}

\section{Proof of Theorem \ref{theorem, total convergence rate}}
Thanks to the triangle inequality, one can bound the subspace estimation error by bounding the error term (i): $\mathbf{ Loss}_1:=\left\|P_{\mc S_{\Y|\X}^{(m)}}-P_{ \widehat {\mc S}_{\Y|\X}^{(m)}}\right\| $ and error term (ii): $\mathbf{ Loss}_2:= \left\|P_{\mathcal S_{\Y|\boldsymbol{X}}}-P_{\mathcal S_{\Y|\boldsymbol{X}}^{(m)}}\right\| $ respectively.
\subsection{Upper bound of error term (i)}
We first give the following lemmas, whose proofs are all deferred to the end of this section.
\begin{lemma}\label{lem:Gammam dagger Mm uniformly bounded}
% Under Assumptions $\ref{as:joint distribution assumption}$ and $\ref{as:Linearity condition and Coverage condition}$,
% $\{\|\Gamma_m^\dagger M_m\| \}_{m=1}^\infty$ is uniformly (about $m$) bounded by $\|\Gamma^{-1}M\| $.
Under Assumptions $\ref{as:joint distribution assumption}$ and $\ref{as:Linearity condition and Coverage condition}$, it holds that $\|\Gamma_m^\dagger M_m\| \leq \|\Gamma^{-1}M\| (\forall m).$
% \begin{align*}
% \|\Gamma_m^\dagger M_m\| \leq \|\Gamma^{-1}M\| \quad\forall m.
% \end{align*}
% $\{\|\Gamma_m^\dagger M_m\| \}_{m=1}^\infty$ is uniformly (about $m$) bounded by $\|\Gamma^{-1}M\| $.
\end{lemma}
\begin{lemma}\label{lem: Gamma inverse M to Gammam dagger Mm}Under Assumptions $\ref{as:joint distribution assumption}$ and $\ref{as:Linearity condition and Coverage condition}$, we have \[\lim\limits_{m\to\infty}\lno\Gamma^{-1}M-\Gamma_m^\dagger M_m\rno =0.\]
\end{lemma}
\noindent We denote by $m_T(\varepsilon)$ the minimal integer $m_T$ satisfying $\lno\Gamma^{- 1}M-\Gamma_m^\dagger M_m\rno \hspace{-1mm}\leqslant \varepsilon$ for all $m\geqslant m_T$ and define an event 
$$\ttE:=\lb \left\|\widehat\Gamma_m^\dagger \widehat M_m^d-\Gamma_m^\dagger M_m\right\|  \leqslant\hspace{-0.5mm}\left(\tfrac{D_0+1}{D_2}\right)^{\frac52}\tfrac{24}{\wt C}n^{c_1(\alpha_1+1)+\gamma-\frac{1}{2}}+D_3n^{\frac{c_1(2\alpha_1+1)-1}{2}}\rb.$$
Then by taking $C$ to be $(D_0+1)n^{\frac{2\gamma}{5}}-\ln\l D_1m^2n \r$ in  Proposition \ref{prop:bound of finite estimate}, one has: for $n\geqslant \l\frac{D_0+1}{D_2}\r^{\frac{5}{1-2\gamma}}$,
$$\P(\ttE)\geq 1-D_1m^2n\exp\left[-(D_0+1)n^{\frac{2\gamma}{5}}\right] -2\exp(- C'm).$$
\begin{lemma}\label{lem:lower bound sigma min total}
Introducing $
\bigtriangleup :=\max\lb \frac{\sigma_d(\Gamma^{-1} M)}{2},\frac{\sigma_d(\Gamma^{-1} M)^2}{4\|\Gamma^{-1}M\| } \rb$.
Suppose that Assumptions $\ref{as:joint distribution assumption}$ to $\ref{assumption: rate-type condition}$ hold, $c_1(2\alpha_1+1)-1<0$ and $2(c_1(\alpha_1+1)+\gamma)-1<0$. Then there exists a positive constant
\begin{align*}
n_2'=n_2'\l\sigma_d(\Gamma^{-1}M),\|\Gamma^{-1}M\| , \gamma,\sigma_0,\sigma_1,\bs K,m_M(1),c_1,m_T\l \tfrac{\bigtriangleup}{2}\r,\wt C,\alpha_1\r
\end{align*}
such that when $n\geqslant n_2'$, we have
\begin{align}
\sigma_{\min}^+(\Gamma_m^{\dagger} M_m)^2\geqslant \tfrac{\sigma_d(\Gamma^{-1}M)^2}{2} \label{eq: lower bound of sigma min}. 
\end{align}
Furthermore, Conditioning on $\ttE$, we have
\begin{align}\label{eq: lower bound of sigma min hat}
&\sigma_{\min}^+(\wh\Gamma_m^{\dagger}\wh M_m^d)^2\geqslant \tfrac{\sigma_d(\Gamma^{-1}M)^2}{2}.
\end{align}
\end{lemma}
The following proposition is an upper bound of error term (i):
\begin{proposition}\label{proposition, estimation error}
Positive constants $D_1$, $D_2$ and $C'$  as in Proposition $\ref{prop:bound of finite estimate}$,
suppose that Assumptions $\ref{as:joint distribution assumption}$ to $\ref{assumption: rate-type condition}$ hold, then $\forall \gamma\in(0,1/2)$, if $c_1$ satisfies $2c_1(\alpha_1+1)+2\gamma-1<0$ and $c_1(2\alpha_1+1)-1<0$, there exists a positive constant $C_1:=C_1\l \|\Gamma^{-1}M\| ,\sigma_d(\Gamma^{-1}M) ,\wt C,\gamma,\sigma_0,\sigma_1\r$ such that
\begin{align*}
\P\l
\lno P_{\mc{S}_{\Y|\X}^{(m)}}-P_{ \widehat{\mc{S}}_{\Y|\X}^{(m)}}\rno \leqslant C_1\frac{m^{\alpha_1+1}}{n^{1/2-\gamma}}\r\geqslant1-2\exp(- C'm)&\\
- D_1m^2n\exp\l -(D_0+1)n^{\frac{2\gamma}{5}} \r&,
\end{align*}
when 
\begin{align*}
n\geqslant\max\Bigg\{ n_1,\l\tfrac{D_0+1}{D_2}\r^{\frac{5}{1-2\gamma}},\left[\tfrac{\|\Gamma^{-1}M\|  \wt C}{48}\l\tfrac{D_2}{D_0+1}\r^{\frac52}\right]^{\frac{2}{2(c_1(\alpha_1+1)+\gamma)-1}}&,\\
\l \tfrac{\|\Gamma^{-1}M\| }{2D_3}\r^{\frac{2}{c_1(2\alpha_1+1)-1}},n_2',\left[ \tfrac{D_3\wt C}{24}\l \tfrac{D_2}{D_0+1} \r^{\frac52} \right]^{\frac2{2\gamma+c_1}}&\Bigg\}
\end{align*}
where $n_2'$ is defined in Lemma $\ref{lem:lower bound sigma min total}$.
\end{proposition}
\begin{proof}
By Lemma $\ref{lemma, way of estimate truncate central subspace}$, $\eqref{def: estimator central subspace}$ and Lemma $\ref{lemma, sin theta of nonadjoint operator}$, one has
\begin{align}
&\left\|P_{\mc S_{\Y|\vX}^{(m)}}-P_{\wh{\mc{S}}_{\Y|\vX}^{(m)}}\right\| =\left\|P_{\Gamma_m^{\dagger}M_m}-P_{\wh\Gamma_m^{\dagger}\wh M_m^d}\right\| \nonumber\\
&\qquad\leqslant\frac{\pi}{2}\frac{\lno\widehat\Gamma_m^\dagger \widehat M_m^d-\Gamma_m^\dagger M_m\rno ^2+\lno\widehat\Gamma_m^\dagger \widehat M_m^d-\Gamma_m^\dagger M_m\rno \lno\Gamma_m^\dagger M_m\rno }{\min\lb\sigma_{\min}^+\l\wh\Gamma_m^\dagger \wh M_m^d\r^2,\sigma_{\min}^+\l\Gamma_m^\dagger M_m\r^2\rb}\label{eq: PS minus P hat S norm}.
% &\leqslant C_5\|\widehat\Gamma_m^\dagger \widehat M_m^d-\Gamma_m^\dagger M_m\|\\
% &=\widetilde O_{\mathbb{P}}\l\frac{m^{\alpha_1+1}}{n^{1/2}}\r,
\end{align}
% with probability at least $1-\exp(- C)-2\exp(- C'm)$.
Because of $c_1(2\alpha_1+1)-1<0$ and $2(c_1(\alpha_1+1)+\gamma)-1<0$, it is easy to check that when
\[n\geqslant\max\lb\left[\tfrac{\|\Gamma^{-1}M\|  \wt C}{48}\l\tfrac{D_2}{D_0+1}\r^{\frac52}\right]^{\frac{2}{2(c_1(\alpha_1+1)+\gamma)-1}},\l \tfrac{\|\Gamma^{-1}M\| }{2D_3}\r^{\frac{2}{c_1(2\alpha_1+1)-1}}\rb,\]
both $\l\tfrac{D_0+1}{D_2}\r^{\frac52}\tfrac{24}{\wt C}n^{c_1(\alpha_1+1)+\gamma-\frac{1}{2}}$ and $D_3n^{\frac{c_1(2\alpha_1+1)-1}{2}}$ are less than or equal to $\frac{\|\Gamma^{-1}M\| }{2}$. Thus, on the event $\ttE$,
\begin{align}\label{eq: high prob upper bound is Gamma minus 1 M}
\lno\widehat\Gamma_m^\dagger \widehat M_m^d-\Gamma_m^\dagger M_m\rno \leqslant \lno\Gamma^{-1}M\rno .
\end{align}
By Lemma $\ref{lem:Gammam dagger Mm uniformly bounded}$, inserting \eqref{eq: high prob upper bound is Gamma minus 1 M} into \eqref{eq: PS minus P hat S norm} leads to
$$
\lno P_{\mc{S}_{\Y|\X}^{(m)}}-P_{ \widehat{\mc{S}}_{\Y|\X}^{(m)}}\rno
\leqslant \frac{\pi\lno\widehat\Gamma_m^\dagger \widehat M_m^d-\Gamma_m^\dagger M_m\rno \lno\Gamma^{-1}M\rno }{\min\lb\sigma_{\min}^+\l\wh\Gamma_m^\dagger \wh M_m^d\r^2,\sigma_{\min}^+\l\Gamma_m^\dagger M_m\r^2\rb},
$$
on the event $\ttE$.
Furthermore, when $n\geqslant \left[ \frac{D_3\wt C}{24}\l \frac{D_2}{D_0+1} \r^{\frac52} \right]^{\frac2{2\gamma+c_1}}$ and $n\geq n_2'$, one can get
$\l \tfrac{D_0+1}{D_2}\r^{\frac52}\tfrac{24m^{\alpha_1+1}}{\wt C n^{1/2-\gamma}}$ is greater than or equal to $D_3\tfrac{m^{(2\alpha_1+1)/2}}{n^{1/2}}$
and then on the event $\ttE$,
\begin{align*}
\lno P_{\mc{S}_{\Y|\X}^{(m)}}-P_{ \widehat{\mc{S}}_{\Y|\X}^{(m)}}\rno \leqslant \tfrac{96\pi\|\Gamma^{-1}M\| }{\sigma_d(\Gamma^{-1}M)^2}\l \tfrac{D_0+1}{D_2}\r^{\frac52}\tfrac{m^{\alpha_1+1}}{\wt C n^{1/2-\gamma}}.
\end{align*}
 by
Lemma $\ref{lem:lower bound sigma min total}$.
Then choosing $C_1=\tfrac{96\pi\|\Gamma^{-1}M\| }{\wt C\sigma_d(\Gamma^{-1}M)^2}\l \tfrac{D_0+1}{D_2}\r^{\frac52}$ can complete the proof.
\end{proof}

\paragraph{Proof of Lemma \ref{lem:Gammam dagger Mm uniformly bounded}}
\begin{proof}
First, it is easy to check that:
\begin{align}
\Gamma^\dag_m=\Pi_m\Gamma^{-1}\Pi_m=\Pi_m\Gamma^{-1}=\Gamma^{-1}\Pi_m=\sum\limits_{i=1}^m\lambda_i^{-1}\phi_i\otimes\phi_i.\label{eq: Gamma m dag def}
\end{align}
According to \eqref{eq: Gamma m dag def} and $M_m=\Pi_mM\Pi_m$, it is easy to check that $\Gamma_m^\dagger M_m=\Pi_m \Gamma^{- 1}M\Pi_m$. Then by the compatibility of operator norm, one can get
\begin{align*}
\lno\Gamma_m^\dagger M_m\rno =\lno\Pi_m \Gamma^{-1}M\Pi_m\rno \leqslant \lno\Pi_m\rno  \lno\Gamma^{-1}M\rno \lno\Pi_m\rno =\lno\Gamma^{-1}M\rno .
\end{align*}
Note that $\Gamma^{-1}M$ is bounded since $\Gamma^{-1}M$ is of finite rank by Corollary $\ref{corollary, MDDO and central subspace}$. Thus the proof is completed. 
\end{proof}

\paragraph{Proof of Lemma \ref{lem: Gamma inverse M to Gammam dagger Mm}}
\begin{proof}
It is easy to check that
$\Gamma_m^\dagger M_m=\Pi_m\Gamma^{-1}M\Pi_m$ and $\Gamma^{-1}M$ is of finite rank by Corollary $\ref{corollary, MDDO and central subspace}$.
Thus the proof is completed by Lemma $\ref{lem:PimTPimtoT}$.
\end{proof}
\paragraph{Proof of Lemma \ref{lem:lower bound sigma min total}}
\begin{proof}
We first prove \eqref{eq: lower bound of sigma min}.
By Corollary $\ref{corollary, MDDO and central subspace}$ and Lemma $\ref{lem:projection equality}$, one has $\rank(\Gamma^{- 1}M)=\rank\l\Gamma^{- 1}M(\Gamma^{- 1}M)^*\r=d$. Thus
\begin{align*}
\sigma_{\min}^+(\Gamma^{-1}M)^2=\lambda_{\min}^+\l\Gamma^{-1}M(\Gamma^{-1}M)^*\r=\lambda_d\l \Gamma^{-1}M(\Gamma^{-1}M)^*\r. 
\end{align*}
 It is easy to see $\rank(\Gamma_m^\dagger M_m)=\rank\l \Gamma_m^\dagger M_m(\Gamma_m^\dagger M_m)^*\r\leqslant d$ by $\Gamma_m^\dagger M_m=\Pi_m \Gamma^{-1} M \Pi_m$ and Lemma $\ref{lem:projection equality}$, thus one can assume that 
 \begin{align*}
\sigma_{\min}^+(\Gamma^\dagger_m M_m)^2=\lambda_{\min}^+\l\Gamma_m^\dagger M_m(\Gamma_m^\dagger M_m)^*\r=\lambda_j\l \Gamma_m^\dagger M_m(\Gamma_m^\dagger M_m)^*\r
\end{align*}
for some $j\leqslant d$.
By Corollary $\ref{coro:wely ineq operator}$, $\eqref{eq:sy ineq}$ and
% (Notice that $M_m$ and $M$ are both compact and self-adjoint)
Lemma $\ref{lem: Gamma inverse M to Gammam dagger Mm}$
%and Lemma \ref{lem:Gammam dagger Mm uniformly bounded}
, one has
\begin{align*}
&\left|\sigma_{\min}^+(\Gamma^\dagger_m M_m)^2\hspace{-0.5mm}-\hspace{-0.5mm}\sigma_j(\Gamma^{-1} M)^2\right|\hspace{-0.5mm}=\hspace{-0.5mm}\left|\lambda_{j}\hspace{-1mm}\l\Gamma^\dagger_m M_m(\Gamma^\dagger_m M_m)^{*}\hspace{-0.5mm}\r\hspace{-0.5mm}-\hspace{-0.5mm}\lambda_j\hspace{-1mm}\l \Gamma^{-1} M(\Gamma^{-1} M)^*\hspace{-0.5mm}\r\right|\\
&\qquad\leqslant
\|\Gamma^{-1} M(\Gamma^{-1} M)^*- \Gamma_m^\dagger M_m(\Gamma_m^\dagger M_m)^*\| \\
&\qquad\leqslant \|\Gamma^{-1} M- \Gamma_m^\dagger M_m\| ^2+
\|\Gamma^{-1} M- \Gamma_m^\dagger M_m\| \cdot\|\Gamma^{-1} M\| \xrightarrow{m\to\infty} 0. 
% &\leqslant\|\Gamma^{-1} M- \Gamma_m^\dagger M_m\|\cdot3\|\Gamma^{-1} M\|
\end{align*}
Thus for 
$
n\geqslant m_T(\bigtriangleup)^{\frac1{c_1}}=m_T\l\max\lb\frac{\sigma_d(\Gamma^{-1} M)}{2},\frac{\sigma_d(\Gamma^{-1} M)^2}{4\|\Gamma^{-1}M\| }\rb\r^{\frac1{c_1}}, 
$
one has $\|\Gamma^{-1} M- \Gamma_m^\dagger M_m\| ^2$ and $\|\Gamma^{-1} M- \Gamma_m^\dagger M_m\| \cdot\|\Gamma^{-1} M\| $ are both less than or equal to $\frac{1}{4}\sigma_d(\Gamma^{-1} M)^2$. Hence one can get
$\left|\sigma_{\min}^+(\Gamma^\dagger_m M_m)^2-\sigma_j(\Gamma^{-1} M)^2\right|\leqslant\frac{1}{2}\sigma_d(\Gamma^{-1} M)^2$
% \begin{align*}\label{eq:sigma min Mm}
% \left|\sigma_{\min}^+(\Gamma^\dagger_m M_m)^2-\sigma_j(\Gamma^{-1} M)^2\right|\leqslant\frac{1}{2}\sigma_d(\Gamma^{-1} M)^2
% \|
% \lambda_j\l \Gamma_m^\dagger M_m\l\Gamma_m^\dagger M_m\r^*\r\geqslant \lambda_j\l \Gamma^{-1} M\l\Gamma^{-1} M\r^*\r-\frac{\lambda_d\l \Gamma^{-1} M\l\Gamma^{-1} M\r^*\r}{2}
% \geqslant\frac{\lambda_d\l \Gamma^{-1} M\l\Gamma^{-1} M\r^*\r}{2}. 
% \end{align*}
and
\begin{equation}
\sigma_{\min}^+(\Gamma^\dagger M_m)^2\geqslant \sigma_j(\Gamma^{-1} M)^2-\frac{1}{2}\sigma_d(\Gamma^{-1} M)^2\geqslant\frac{1}{2}\sigma_d(\Gamma^{-1} M)^2
\end{equation}
for sufficiently large $n$. This completes the proof of \eqref{eq: lower bound of sigma min}.

Next we prove $\eqref{eq: lower bound of sigma min hat}$. Combining Proposition $\ref{prop:bound of finite estimate}$ with Lemma $\ref{lem: Gamma inverse M to Gammam dagger Mm}$ leads to that on the event $\ttE$, 
$$
\lno\wh \Gamma_m^\dag\wh M^d_m- \Gamma^{-1}M\rno \leqslant\ve+\l\tfrac{D_0+1}{D_2}\r^{\frac52}\tfrac{24}{\wt C}n^{c_1(\alpha_1+1)+\gamma-\frac{1}{2}}+D_3n^{\frac{c_1(2\alpha_1+1)-1}{2}}
$$
for  $n\geqslant \max\{n_1,m_T( \ve)^{1/c_1}\}$.
Assuming that $c_1(2\alpha_1+1)-1<0$ and $2(c_1(\alpha_1+1)+\gamma)-1<0$, it is easy to check that when $$n\geqslant\max\lb\left[\frac{\bigtriangleup \wt C}{96}\l\frac{D_2}{D_0+1}\r^{\frac52}\right]^{\frac{2}{2(c_1(\alpha_1+1)+\gamma)-1}},\l \frac{\bigtriangleup}{4D_3}\r^{\frac{2}{c_1(2\alpha_1+1)-1}}\rb$$, both $\l\tfrac{D_0+1}{D_2}\r^{\frac52}\tfrac{24}{\wt C}n^{c_1(\alpha_1+1)+\gamma-\frac{1}{2}}$ and $D_3n^{\frac{c_1(2\alpha_1+1)-1}{2}}$ are less than or equal to $\frac{\bigtriangleup}{4}$. Letting $\varepsilon=\frac12\bigtriangleup$, one can get on the event $\ttE$,
when
\begin{align*}
n&\geqslant n_2'=n_2'\hspace{-0.5mm}\l\hspace{-0.5mm}\sigma_d(\Gamma^{-1}M),\|\Gamma^{-1}M\| , \gamma,\sigma_0,\sigma_1,\bs K,m_M(1),c_1,m_T\l \tfrac{\bigtriangleup}{2}\r,\wt C,\alpha_1\hspace{-0.5mm}\r\\
&:=\max\bigg\{ n_1,m_T\l \tfrac{\bigtriangleup}{2}\r^{1/c_1}, \left[\tfrac{\bigtriangleup \wt C}{96}\l\tfrac{D_2}{D_0+1}\r^{\frac52}\right]^{\frac{2}{2(c_1(\alpha_1+1)+\gamma)-1}},\l \tfrac{\bigtriangleup}{4D_3}\r^{\frac{2}{c_1(2\alpha_1+1)-1}}\bigg\},
\end{align*}
one has $\lno\wh \Gamma_m^\dag\wh M^d_m- \Gamma^{-1}M\rno \leqslant\bigtriangleup$ and further
$\sigma_{\min}^+(\wh\Gamma^\dagger \wh M^d_m)^2\hspace{-1mm}\geqslant\hspace{-1mm} \tfrac{\sigma_d(\Gamma^{-1} M)^2}{2}$ by the same argument as the proof of \eqref{eq: lower bound of sigma min}.
 This completes the proof of \eqref{eq: lower bound of sigma min hat}.
Considering that $m_T(\bigtriangleup)\leqslant m_{T}\l\frac\bigtriangleup2\r$, one can also get $\eqref{eq: lower bound of sigma min}$ when $n\geqslant n_2'$. Thus the proof is completed.
\end{proof}
\subsection{Upper bound of error term (ii)}\label{ap, subs, truncation error}
\begin{proposition}\label{proposition, truncation error}
Under Assumption $\ref{assumption: rate-type condition}$, there exists a positive constant $C_2:=C_2\l d,\wt C,\lambda_d(\mc{B}),\alpha_2\r$ where $\mc{B}:=\sum\limits_{i=1}^d {\bs{\beta}}_i\otimes{\bs{\beta}}_i$ for ${\bs{\beta}}_i$ defined in \eqref{def: central subspace}, such that when $n\geqslant \l \frac{\lambda_d({\mc{B}})}{4d\wt C^2}\sqrt{\frac{2\alpha_2-1}{\zeta(2\alpha_2)}}\r^{\frac{2}{c_1(1-2\alpha_2)}}$, we have
\begin{equation}\label{equation, truncation error}
 \left\|P_{\mathcal S_{\Y|\boldsymbol{X}}}-P_{\mathcal S_{\Y|\boldsymbol{X}}^{(m)}}\right\| \leqslant C_2m^{-\frac{2\alpha_2-1}{2}},
\end{equation}
where $\zeta(\cdot)$ is Riemann $\zeta$ function.
% \begin{equation}
% \|P_{\mathcal S_{Y|\boldsymbol{X}}}-P_{\mathcal S_{Y|\boldsymbol{X}}^{(m)}}\|\leqslant O_{\mathbb{P}}(dn^{-(\alpha_2-1)/(2\alpha_1+\alpha_2)}) 
% \end{equation}
\end{proposition}
\begin{proof}
Let ${\mc{B}^{(m)}}:=\sum\limits_{i=1}^d {\bs{\beta}}_i^{(m)}\otimes{\bs{\beta}}_i^{(m)}$ for ${\bs{\beta}}_i^{(m)}$ defined in \eqref{def: truncated central subspace}.
Combing with Equation $\eqref{def: central subspace}$, it is easy to check that $\left\|P_{\mathcal S_{\Y|\boldsymbol{X}}}-P_{\mathcal S_{\Y|\boldsymbol{X}}^{(m)}}\right\| =\|P_{\mc{B}}-P_{\mc{B}^{(m)}}\| $. By Corollary $\ref{cor: sin theta self adjoint}$, we have
\begin{align}\label{eq:sin theta for B Bm}
\|P_{\mc{B}}-P_{\mc{B}^{(m)}}\| \leqslant \frac{\pi}{2}\frac{\|{\mc{B}}-{\mc{B}^{(m)}}\| }{\min\{\lambda_{\min}^+({\mc{B}}),\lambda_{\min}^+({\mc{B}^{(m)}})\}}.
\end{align}

Note that ${\mc{B}}-{\mc{B}^{(m)}}$ is self-adjoint, then
\begin{align*}
&\lno{\mc{B}}-{\mc{B}^{(m)}}\rno =\sup_{{\bs{\beta}}\in\mathbb{S}_{ \mathcal H}}|\langle ({\mc{B}}-{\mc{B}^{(m)}})({\bs{\beta}}),{\bs{\beta}}\rangle|=\sup_{{\bs{\beta}}\in\mathbb{S}_{\mathcal H}}|\langle {\mc{B}}{\bs{\beta}},{\bs{\beta}}\rangle-\langle {\mc{B}^{(m)}}{\bs{\beta}},{\bs{\beta}}\rangle|\\
&~~=\sup_{{\bs{\beta}}\in\mathbb{S}_{\mathcal H}}\hspace{-0.9mm}\left|\sum_{i=1}^d\hspace{-0.9mm}\left[\langle{\bs{\beta}}_i,{\bs{\beta}}\rangle^2-\langle{\bs{\beta}}_i^{(m)},{\bs{\beta}}\rangle^2\right]\right|=\sup_{{\bs{\beta}}\in\mathbb{S}_{\mathcal H}}\hspace{-0.9mm}\left| \sum_{i=1}^d\langle{\bs{\beta}}_i-{\bs{\beta}}_i^{(m)},{\bs{\beta}}\rangle\langle{\bs{\beta}}_i+{\bs{\beta}}_i^{(m)},{\bs{\beta}}\rangle\right|\\
&~~\leqslant\sup_{{\bs{\beta}}\in\mathbb{S}_{\mathcal H}}\sum_{i=1}^d\left| \langle{\bs{\beta}}_i-{\bs{\beta}}_i^{(m)},{\bs{\beta}}\rangle\langle{\bs{\beta}}_i+{\bs{\beta}}_i^{(m)},{\bs{\beta}}\rangle\right|
\leqslant\sum_{i=1}^d\left\|{\bs{\beta}}_i-{\bs{\beta}}_i^{(m)}\right\|\left\|{\bs{\beta}}_i+{\bs{\beta}}_i^{(m)}\right\|,
\end{align*}
where the first inequality comes from the triangle inequality, and the 
second inequality comes from the Cauchy-Schwarz inequality and $\|{\bs{\beta}}\|=1$. 
 Then one has ${\bs{\beta}}_i=\sum\limits_{j=1}^\infty b_{ij}\phi_j$ and 
\[{\bs{\beta}}^{(m)}_i=\Pi_m{\bs{\beta}}_i=\sum_{j'=1}^m\phi_{j'}\otimes\phi_{j'}\sum_{j=1}^\infty b_{ij}\phi_j=\sum_{j'=1}^m\sum_{j=1}^\infty\langle\phi_{j'},\phi_j\rangle b_{ij}\phi_{j'}=\sum_{j=1}^mb_{ij}\phi_j.\]
According to Assumption $\ref{assumption: rate-type condition}$, one can get
\begin{align*}
\left\|{\bs{\beta}}_i-{\bs{\beta}}_i^{(m)}\right\|&=\left\|\sum_{j=m+1}^\infty b_{ij}\phi_j\right\|=\sqrt{\sum_{j=m+1}^\infty b_{ij}^2}\leqslant \wt C\sqrt{\sum_{j=m+1}^\infty j^{-2\alpha_2}};\\
\left\|{\bs{\beta}}_i+{\bs{\beta}}_i^{(m)}\right\|&\leqslant\|{\bs{\beta}}_i\|+\lno{\bs{\beta}}_i^{(m)}\rno\leqslant2\|{\bs{\beta}}_i\|=2\sqrt{\sum_{j=1}^\infty b_{ij}^2}\leqslant 2\wt C\sqrt{\sum_{j=1}^\infty j^{- 2\alpha_2}}.
\end{align*}
Because $\alpha_2>1/2$, one has
\[\sum\limits_{j=m+1}^\infty \frac{1}{j^{2\alpha_2}}\leqslant \frac{1}{2\alpha_2-1}\frac{1}{m^{2\alpha_2-1}};\qquad \sum_{j=1}^\infty \frac 1{j^{2\alpha_2}}=\zeta(2\alpha_2)\text{ is convergent},\]
where $\zeta(\cdot)$ is Riemann $\zeta$ function. Thus, one can get
\begin{equation}\label{eq: upper bound of operator norm of A minus B}
\lno{\mc{B}}-{\mc{B}^{(m)}}\rno \leqslant 2d\wt C^2\sqrt{\frac{\zeta(2\alpha_2)}{2\alpha_2-1}}m^{-\frac{2\alpha_2-1}{2}}.
\end{equation}

Furthermore, 
{since $\mr{rank}(\mc{B})=d$, one can get that $\lambda_{\min}^+(\mc{B})=\lambda_{d}(\mc{B})$. It is easy to see $\rank(\mc{B}^{(m)})\leqslant d$ by $\mc{B}^{(m)}=\Pi_m \mc{B} \Pi_m$, thus one can assume that $\lambda_{\min}^+(\mc{B}^{(m)})=\lambda_j( \mc{B}^{(m)})$ for some $j\leqslant d$.
By Corollary $\ref{coro:wely ineq operator}$
% (Notice that $M_m$ and $M$ are both compact and self-adjoint)
and \eqref{eq: upper bound of operator norm of A minus B}, one has:
$$
|\lambda_j( \mc{B}^{(m)})-\lambda_j\l \mc{B}\r|\leqslant\lno \mc{B}-\mc{B}^{(m)}\rno \leqslant 2d\wt C^2\sqrt{\frac{\zeta(2\alpha_2)}{2\alpha_2-1}}m^{-\frac{2\alpha_2-1}{2}}.
$$
Thus for sufficiently large {$n\geqslant \l \frac{\lambda_d({\mc{B}})}{4d\wt C^2}\sqrt{\frac{2\alpha_2-1}{\zeta(2\alpha_2)}} \r^{\frac{2}{c_1(1-2\alpha_2)}}$}, one has
\begin{align}
&\lambda_j\l \mc{B}^{(m)}\r\geqslant \lambda_j\l \mc{B}\r-\frac{\lambda_d\l \mc{B}\r}{2}
\geqslant\frac{\lambda_d\l \mc{B}\r}{2}\nonumber\\
&\qquad\Longrightarrow \min\{\lambda_{\min}^+({\mc{B}}),\lambda_{\min}^+({\mc{B}^{(m)}})\}\geqslant \frac{\lambda_d({\mc{B}})}{2}. \label{eq:lower bound lambda min plus B Bm}
\end{align}}
Inserting \eqref{eq: upper bound of operator norm of A minus B} and \eqref{eq:lower bound lambda min plus B Bm} into \eqref{eq:sin theta for B Bm} leads to
\begin{align*}
\left\|P_{\mathcal S_{\Y|\boldsymbol{X}}}-P_{\mathcal S_{\Y|\boldsymbol{X}}^{(m)}}\right\| \leqslant \frac{2\pi d\wt C^2}{\lambda_{d}(\mc{B})}\sqrt{\frac{\zeta(2\alpha_2)}{2\alpha_2-1}}m^{-\frac{2\alpha_2-1}{2}}.
\end{align*}
Then choosing $C_2:=\frac{2\pi d\wt C^2}{\lambda_d({\mc{B}})}\sqrt{\frac{\zeta(2\alpha_2)}{2\alpha_2-1}}$ can complete the proof.
\end{proof}

\subsection{Proof of Theorem \ref{theorem, total convergence rate}}
\begin{proof}
Note that
\begin{equation}
\begin{aligned}
\left\|P_{\mc{S}_{\Y|\X}}-P_{\widehat{\mc{S}}_{\Y|\X}^{(m)}}\right\| 
&\leqslant \left\|P_{\mc{S}_{\Y|\X}}-P_{\mc{S}_{\Y|\X}^{(m)}}\right\| +\left\|P_{\mc{S}_{\Y|\X}^{(m)}}-P_{ \widehat{\mc{S}}_{\Y|\X}^{(m)}}\right\| .\\
\end{aligned}
\end{equation}
Next we select $m$ to be $n^{\frac{1-2\gamma}{2\alpha_1+2\alpha_2+1}}$, i.e.,  $c_1:=\frac{1-2\gamma}{2\alpha_1+2\alpha_2+1}$. And it is easy to check that $c_1$ satisfies $2c_1(\alpha_1+1)+2\gamma-1=-\frac{(1-2\gamma)(2\alpha_2-1)}{2\alpha_1+2\alpha_2+1}<0$ and $c_1(2\alpha_1+1)-1=-\frac{2[\gamma(2\alpha_1+1)+\alpha_2]}{2\alpha_1+2\alpha_2+1}<0$.
Then combining Proposition $\ref{proposition, estimation error}$ with Proposition $\ref{proposition, truncation error}$ leads to
\begin{align*}
\P\left[\left\|P_{\mc S_{\Y|\X}}-P_{\widehat{\mc{S}}_{\Y|\X}^{(m)}}\right\| \leqslant\hspace{-0.5mm} (C_1+C_2)n^{-\frac{(2\alpha_2-1)(1-2\gamma)}{2(2\alpha_1+2\alpha_2+1)}}\right]\hspace{-1mm}\geqslant\hspace{-1mm} 1-2\exp\hspace{-0.5mm}\l\hspace{-1mm}- C'n^{\frac{1-2\gamma}{2\alpha_1+2\alpha_2+1}}\r&\\
-\exp\left[\ln\l D_1n^{\frac{2\alpha_1+2\alpha_2+3-4\gamma}{2\alpha_1+2\alpha_2+1}} \r-(D_0+1)n^{\frac{2\gamma}{5}}\right]&
\end{align*}
when $n\geqslant n_3'$, where
\begin{align*}
n_3'=\max\Bigg\{n_1,n_2',\left[\tfrac{\|\Gamma^{-1}M\|  \wt C}{48}\l\tfrac{D_2}{D_0+1}\r^{\frac52}\right]^{\frac{2}{2(c_1(\alpha_1+1)+\gamma)-1}}\hspace{-0.9mm},\l \tfrac{\|\Gamma^{-1}M\| }{2D_3}\r^{\frac{2}{c_1(2\alpha_1+1)-1}}\hspace{-0.9mm},\\
\l\tfrac{D_0+1}{D_2}\r^{\frac{5}{1-2\gamma}},\left[ \tfrac{D_3\wt C}{24}\l \tfrac{D_2}{D_0+1} \r^{\frac52} \right]^{\frac2{2\gamma+c_1}},\l\tfrac{\lambda_d(\mc{B})}{4d\wt C^2}\sqrt{\tfrac{{2\alpha_2-1}}{\zeta(2\alpha_2)}} \r^{\frac{2}{c_1(1-2\alpha_2)}}\Bigg\}
\end{align*}

It is easy to check that as long as $\frac{2\gamma}{5}<\frac{1-2\gamma}{2\alpha_1+2\alpha_2+1}\Longrightarrow\gamma<\frac{5}{4(\alpha_1+\alpha_2+3)}$, 
there exists a constant $n_3''=n_3''\l \gamma,\alpha_1,\alpha_2,D_0,D_1,C'\r$ such that when $n\geqslant n_3'$ further, we have 
\begin{align*}
\P\l\left\|P_{\mc S_{\Y|\X}}-P_{\widehat{\mc{S}}_{\Y|\X}^{(m)}}\right\| \leqslant (C_1+C_2)n^{-\frac{(2\alpha_2-1)(1-2\gamma)}{2(2\alpha_1+2\alpha_2+1)}} \r
\geqslant1-2\exp\l-\tfrac{D_0+1}{2}n^{\frac{2\gamma}{5}} \r.
\end{align*}
Thus one can choose $n_3=\max\{n_3',n_3''\}$ to get the following conclusion.
\begin{proposition}
Under Assumptions $\ref{as:joint distribution assumption}$ to $\ref{assumption: rate-type condition}$, for any $\gamma\in\l0,\tfrac{5}{4(\alpha_1+\alpha_2+3)}\r$, choosing 
$m=n^{\frac{1-2\gamma}{2\alpha_1+2\alpha_2+1}}$ (i.e.,  $c_1=\frac{1-2\gamma}{2\alpha_1+2\alpha_2+1}$) yields a positive constant
\begin{align*}
D_4:=D_4\l \|\Gamma^{-1}M\| ,\sigma_d(\Gamma^{-1}M) ,\gamma,\sigma_0,\sigma_1,d,\wt C,\lambda_d\l\sum\limits_{i=1}^d {\bs{\beta}}_i\otimes{\bs{\beta}}_i\r,\alpha_2\r 
\end{align*}
such that when $n$ is sufficiently large, we have:
\begin{align*}
\P\l\left\|P_{\mc{S}_{\Y|\X}}-P_{\widehat{\mc{S}}_{\Y|\X}^{(m)}}\right\| \leqslant D_4n^{-\frac{(2\alpha_2-1)(1-2\gamma)}{2(2\alpha_1+2\alpha_2+1)}} \r
\geqslant1-2\exp\l -\tfrac{D_0+1}{2}n^{\frac{2\gamma}{5}} \r,
\end{align*}
where $D_0$ and $D_1$ are defined in Proposition $\ref{prop:bound hatMmd Mm}$.
\end{proposition}
\noindent
% Theorem $\ref{theorem, total convergence rate}$ is a direct corollary of above proposition.
Define 
$$\mathtt F:=\left\{\left\|P_{\mc{S}_{\Y|\X}}-P_{\widehat{\mc{S}}_{\Y|\X}^{(m)}}\right\| \leqslant D_4n^{-\frac{(2\alpha_2-1)(1-2\gamma)}{2(2\alpha_1+2\alpha_2+1)}}\right\}.$$
Then 
\begin{align*}
 \mb E\left[\left\|P_{\mc{S}_{\Y|\X}}-P_{\widehat{ \mc{S}}_{\Y|\X}^{(m)}}\right\|^2\right] =&
  \mb E\left[\left\|P_{\mc{S}_{\Y|\X}}-P_{\widehat{ \mc{S}}_{\Y|\X}^{(m)}}\right\|^21_{\mathtt{F}}\right] +
   \mb E\left[\left\|P_{\mc{S}_{\Y|\X}}-P_{\widehat{ \mc{S}}_{\Y|\X}^{(m)}}\right\|^21_{\mathtt{F}^c}\right]\\ 
 \leqslant &
 D_4^2n^{-\frac{(2\alpha_2-1)(1-2\gamma)}{2\alpha_1+2\alpha_2+1}}+4\mb P\left( \mathtt F^c\right)\\
 \lesssim&n^{-\frac{(2\alpha_2-1)(1-2\gamma)}{2\alpha_1+2\alpha_2+1}}+\exp\l -\tfrac{D_0+1}{2}n^{\frac{2\gamma}{5}} \r\\
\lesssim&n^{-\frac{(2\alpha_2-1)(1-2\gamma)}{2\alpha_1+2\alpha_2+1}}.
\end{align*}
This completes the proof of  Theorem \ref{theorem, total convergence rate}.
\end{proof}

\section{Additional Simulation Results of Section \ref{sec:Synthetic}}
This section contains the additional  simulation results  of Sections \ref{sec:Synthetic}  when $\varepsilon\sim N(0,1)$.

We show the average $\mc D(\bs B;\bs{\wh B})$ with different $m$ or $\rho$ for three methods under $\mc M_1$ to $\mc M_3$ in Figure \ref{fig:error 3models,noise1},
where we mark minimal error in each model with red `$\times$'. The shaded areas represent the standard error associated with these estimates and all of them are less than  $0.01$. For FSFIR, the  minimal errors for $\mc M_1-\mc M_3$ are  $0.08,0.02,0.01$ respectively.
For TFSIR, the  minimal errors are  $0.08,0.02,0.01$ and for regularized FSIR,  the  minimal errors are $0.13,0.06,0.01$.  

\begin{figure}[H]% [H] is so declass\'e!
	\centering
	\begin{minipage}{0.33\textwidth}
		\includegraphics[width=\textwidth]{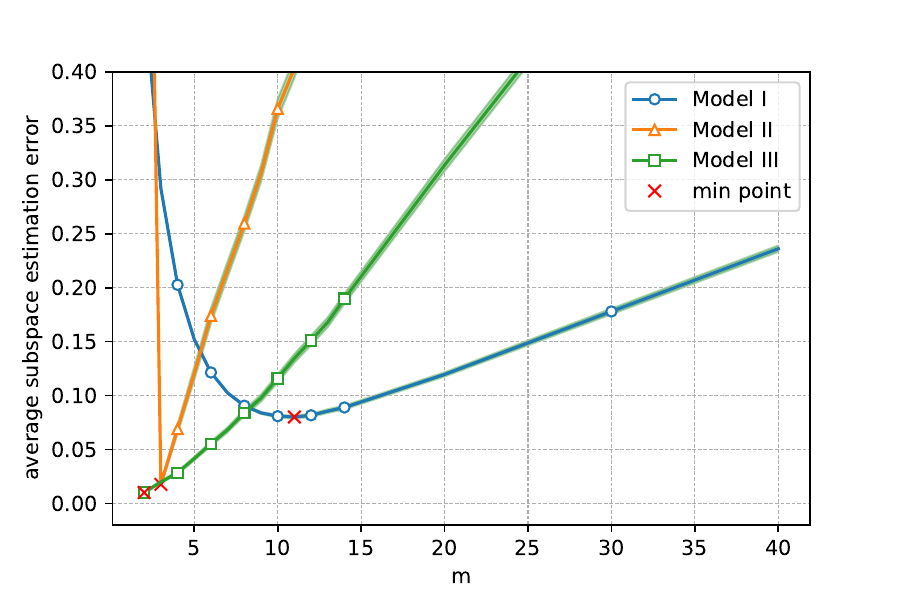}
	\end{minipage}\hfill
	\begin{minipage}{0.33\textwidth}
		\includegraphics[width=\textwidth]{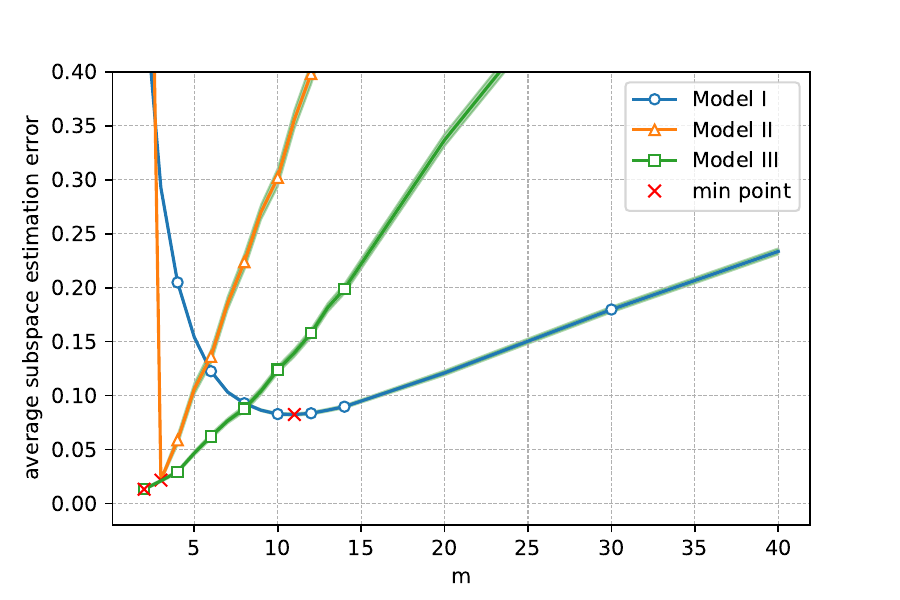}
	\end{minipage}
	\begin{minipage}{0.33\textwidth}
		\includegraphics[width=\textwidth]{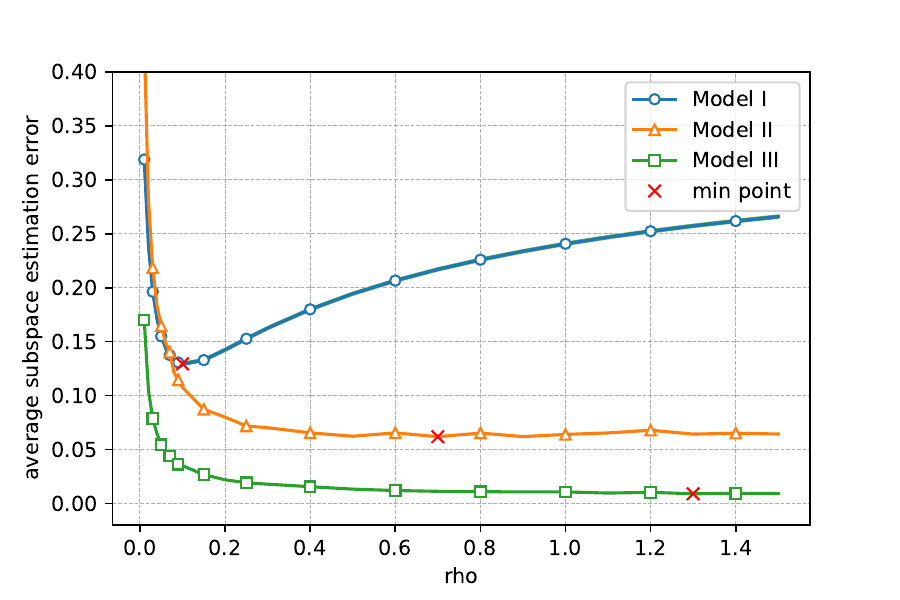}
	\end{minipage}
\caption{Average subspace estimation error of FSFIR (left), TFSIR (middle) and RFSIR (right) for various models. The standard errors are all below $0.01$. }
\label{fig:error 3models,noise1}
\end{figure}

Figure \ref{fig:error 3models,noise1} shows that FSFIR attains the best performance among  all models. 
Moreover, FSFIR is easier to practice as it does not need a slice number $H$ in advance.